\documentclass[reqno,11pt]{article}
\usepackage{mathrsfs}
\usepackage{amssymb}

\usepackage{amsthm}
\usepackage{amsmath}
\usepackage{amsfonts}
\usepackage{enumerate}
\usepackage{bm}

\usepackage{multicol}

\usepackage{subfigure}
\usepackage{graphicx}
\usepackage{setspace,color,wrapfig}
\usepackage[colorlinks,linkcolor=blue,citecolor=cyan]{hyperref}

\numberwithin{equation}{section}
\newtheorem{theorem}{Theorem}[section]
\newtheorem{lemma}[theorem]{Lemma}
\newtheorem{corollary}[theorem]{Corollary}
\newtheorem{proposition}[theorem]{Proposition}

\theoremstyle{definition}
\newtheorem{definition}[theorem]{Definition}
\newtheorem{assumption}[theorem]{Assumption}
\newtheorem{example}[theorem]{Example}

\theoremstyle{remark}
\newtheorem{remark}[theorem]{Remark}
\begin{document}
\title{Subsonic Euler flows in a three-dimensional finitely long  cylinder with arbitrary cross section}
\author{Shangkun Weng\thanks{School of Mathematics and Statistics, Wuhan University, Wuhan, Hubei Province, 430072, People's Republic of China. Email: skweng@whu.edu.cn}
\and Changkui Yao\thanks{School of Mathematics and Statistics, Wuhan University, Wuhan, Hubei Province, 430072, People's Republic of China. Email: Yaochangkui@whu.edu.cn}}
\date{}
\maketitle

\def\be{\begin{eqnarray}}
\def\ee{\end{eqnarray}}
\def\ba{\begin{aligned}}
\def\ea{\end{aligned}}
\def\bay{\begin{array}}
\def\eay{\end{array}}
\def\bca{\begin{cases}}
\def\eca{\end{cases}}
\def\p{\partial}
\def\hphi{\hat{\phi}}
\def\bphi{\bar{\phi}}
\def\no{\nonumber}
\def\eps{\epsilon}
\def\de{\delta}
\def\De{\Delta}
\def\om{\omega}
\def\Om{\Omega}
\def\f{\frac}
\def\th{\theta}
\def\vth{\vartheta}
\def\la{\lambda}
\def\lab{\label}
\def\b{\bigg}
\def\var{\varphi}
\def\na{\nabla}
\def\ka{\kappa}
\def\al{\alpha}
\def\La{\Lambda}
\def\ga{\gamma}
\def\Ga{\Gamma}
\def\ti{\tilde}
\def\wti{\widetilde}
\def\wh{\widehat}
\def\ol{\overline}
\def\ul{\underline}
\def\Th{\Theta}
\def\si{\sigma}
\def\Si{\Sigma}
\def\oo{\infty}
\def\q{\quad}
\def\z{\zeta}
\def\co{\coloneqq}
\def\eqq{\eqqcolon}
\def\di{\displaystyle}
\def\bt{\begin{theorem}}
\def\et{\end{theorem}}
\def\bc{\begin{corollary}}
\def\ec{\end{corollary}}
\def\bl{\begin{lemma}}
\def\el{\end{lemma}}
\def\bp{\begin{proposition}}
\def\ep{\end{proposition}}
\def\br{\begin{remark}}
\def\er{\end{remark}}
\def\bd{\begin{definition}}
\def\ed{\end{definition}}
\def\bpf{\begin{proof}}
\def\epf{\end{proof}}
\def\bex{\begin{example}}
\def\eex{\end{example}}
\def\bq{\begin{question}}
\def\eq{\end{question}}
\def\bas{\begin{assumption}}
\def\eas{\end{assumption}}
\def\ber{\begin{exercise}}
\def\eer{\end{exercise}}
\def\mb{\mathbb}
\def\mbR{\mb{R}}
\def\mbZ{\mb{Z}}
\def\mc{\mathcal}
\def\mcS{\mc{S}}
\def\ms{\mathscr}
\def\lan{\langle}
\def\ran{\rangle}
\def\lb{\llbracket}
\def\rb{\rrbracket}
\def\fr#1#2{{\frac{#1}{#2}}}
\def\dfr#1#2{{\dfrac{#1}{#2}}}
\def\u{{\textbf u}}
\def\v{{\textbf v}}
\def\w{{\textbf w}}
\def\d{{\textbf d}}
\def\nn{{\textbf n}}
\def\x{{\textbf x}}
\def\e{{\textbf e}}
\def\D{{\textbf D}}
\def\U{{\textbf U}}
\def\M{{\textbf M}}
\def\F{{\mathcal F}}
\def\I{{\mathcal I}}
\def\W{{\mathcal W}}
\def\div{{\rm div\,}}
\def\curl{{\rm curl\,}}
\def\R{{\mathbb R}}
\def\FF{{\textbf F}}
\def\A{{\textbf A}}
\def\R{{\textbf R}}
\def\r{{\textbf r}}
\def\mc{\mathcal{C}}

\def\e{\bar}

\begin{abstract}
  This paper concerns the well-posedness of subsonic flows in a three-dimensional finitely long  cylinder with arbitrary cross section.  We establish the existence and uniqueness of subsonic flows in the Sobolev space  by prescribing the normal component of the momentum, the vorticity,  the entropy, the Bernoulli's quantity  at the entrance and the normal component of the momentum at the exit. One of the key points in the analysis is to utilize the deformation-curl decomposition for the steady Euler system introduced in \cite{WX19} to deal with the hyperbolic and elliptic modes. Another one is to  employ the  separation of variables to improve the regularity of solutions to a deformation-curl system near the intersection between the entrance and exit with the cylinder wall.
\end{abstract}
\begin{center}
\begin{minipage}{5.5in}
Mathematics Subject Classifications 2020: 35M12, 76G25, 76N10.\\
Key words: steady Euler system,  structural stability, the deformation-curl decomposition,  Sobolev space, the separation of variables.
\end{minipage}
\end{center}

\section{Introduction and main results}\noindent

\par The three-dimensional steady full Euler equations for compressible inviscid gas are
\begin{align}\label{euler}
\begin{cases}
\p_1(\rho u_1)+\p_2(\rho u_2) +\p_3(\rho u_3)=0,\\
\p_1(\rho u_1^2 + P)+\p_2(\rho u_1 u_2)+\p_3(\rho u_1 u_3) = 0,\\
\p_1(\rho u_1u_2)+\p_2(\rho u_2^2 + P)+\p_3(\rho u_2 u_3) = 0,\\
\p_1(\rho u_1u_3)+\p_2(\rho u_2u_3 )+\p_3(\rho u_3^2 +P) = 0,\\
\text{div }(\rho (\frac12|{\bf u}|^2 +e) {\bf u}+ P{\bf u})=0, \\
\end{cases}
\end{align}
where   \(\textbf{u} = (u_1,u_2,u_3)\), \(\rho\) , \( P\), and \(e\)  stand for the velocity, density, pressure and internal energy , respectively. For polytropic gases, the equation of state and the internal energy are
\begin{align*}
     P = K(S)\rho^{\gamma} , \quad K(S) = ae^{\frac{S}{c_v}}\quad and\quad e = \frac{P}{(\gamma-1)\rho},
\end{align*}
respectively, where \(\gamma > 1\), \( a\), and \(c_v\) are positive constants, and \(S\) is called the specific entropy.
Denote  the local sound speed $c(\rho, K)=\sqrt{\partial_\rho P(\rho, K)}=\sqrt{\gamma K(S)\rho^{\gamma-1}}$. Then  the subsonic flow means that $ |\textbf{u}|<c(\rho, K) $.

\par There have been some  studies on the steady  incompressible flows in bounded domains. Alber \cite{AH92} first established the existence and uniqueness of three dimensional steady incompressible flows with nonzero vorticity on bounded domains with smooth boundary. Tang-Xin \cite{TX09} established the well-posedness of the solutions to the incompressible steady Euler equations with a different boundary condition $\text{curl } \mathbf{u} = a \mathbf{u} + b$ for some functions  $a$ and $b$. It is worthy to point out that background flow in \cite{TX09} was not assumed to satisfy the Euler equations. Molinet \cite{LM99} extended Alber's method to nonsmooth domains and steady compressible  flows. However, he considered domains with a boundary consisting of smooth parts which meet at an angle smaller than 2$\pi$/7 and any integer multiple of the angle may not equal $\pi$. Buffoni and Wahlen \cite{BE19} proved the existence of a class of rotational flows with the form $\nabla f \times \nabla g$ for some functions $f$ and $g$ to the steady incompressible Euler equations in an unbounded domain of the form $(0, L) \times \mathbb{R}^2$, where the flows are periodic in both two unbounded directions.
Seth \cite{SS21} proved an existence result for solutions to the steady incompressible Euler equations in a domain with nonsmooth boundary. The boundary consists of three parts, one where fluid flows into the domain, one where the fluid flows out, and one which no fluid passes through. These three parts meet at right angles. To investrgate the regularity of the div-curl system with normal boundary condition near the corner-points, the author in \cite{SS21} employed the vector potential formulation and reduced the problem to the regularity of the Poisson equation with the mixed boundary conditions.

\par  There also have many literatures on the study of steady compressible subsonic Euler flows in finitely long nozzles. Du-Weng-Xin \cite{DWX14} and Weng \cite{WS14} established the well-posedness  of the subsonic irrotational flows through a two-dimensional finitely long nozzle by characterizing a class of physical acceptable boundary conditions. Chen-Xie  \cite{CX14} obtained the existence and uniqueness of three dimensional steady subsonic Euler flows in rectangular nozzles when prescribing normal component of momentum at both the entrance and exit. Weng \cite{WS15} introduced a reformulation of the steady Euler equations in terms of the reduced velocity, the Bernoulli's quantity and pressure and constructed a smooth subsonic Euler flow in a rectangular cylinder. Later on, Weng-Xin \cite{WX19} and Weng \cite{WS19} had put forward a deformation-curl decomposition for the three dimensional steady Euler system and steady Euler-Poisson system. As an application, they established the structural stability of 1-D subsonic flows under multidimensional perturbations on the entrance and exit of a rectangular cylinder, and the solutions obtained had optimal regularity for all physical quantities, the velocity, the Bernoulli's quantity, the entropy and the pressure share the same regularity.

\par In this paper, we will construct a subsonic solution to \eqref{euler} in a three-dimensional cylinder by imposing suitable boundary conditions at the entrance and exit, which is also close to the background solution $ \bar {\bm U} :=( \bar \rho,\bar u,0,0,\bar K) $, where \(\bar\rho, \;\bar u,\;\)and \(\bar K\)  are three positive constants satisfying the subsonic condition \( 0 < \bar u^2 < \bar K\gamma\bar\rho^{\gamma-1}\).  We also denote the Bernoulli's constant by \( \bar B = \frac{1}{2}|\bar u|^2 + \frac{\gamma}{\gamma - 1}\bar K \bar\rho^{\gamma -1}\).

\par The three-dimensional cylinder is given by
\begin{equation*}
\Omega=\{(x_1,x_2,x_3): x_1 \in (0,L),(x_2,x_3) \in \Sigma\},
\end{equation*}
where \(\Sigma \subset \mbR^2\) is a bounded simply-connected domain with smooth boundary \( \p \Sigma\).  the entrance, wall, and the exit of $\Omega$ are denoted by
\begin{equation*}
    \Gamma_0 = \Omega \cap \{x_1 = 0 \}, \quad \Gamma_w = (0,L) \times \p \Sigma, \quad \Gamma_L = \Omega \cap \{x_1=L\}.
\end{equation*}

First, at the entrance $  \Gamma_0 $, we impose the normal component of the momentum, the vorticity, the Bernoulli's quantity and the entropy :
\begin{align}\label{1-2}
    \begin{cases}
        (\rho u_1)(0,x^{\prime}) = \bar\rho\bar u + \sigma m_{0}(x^{\prime}),\\
    \omega_1(0,x^{\prime}) = \sigma J_{0}(x^{\prime}),\\
    B(0,x^{\prime}) = \Bar{B} + \sigma B_0 (x^{\prime}),\\
    K(0,x^{\prime}) = \Bar{K} + \sigma K_0(x^{\prime}),
    \end{cases}
\end{align}
where $ x^{\prime} = (x_2,x_3) $ and $ m_0 \in H^{\frac{5}{2}}(\Sigma)$ and \(J_{0} \in H^2_0(\Sigma)\) and $ (B_0,K_0)\in (H_0^3(\Sigma))^2 $.
\par At the exit  $  \Gamma_L $, we prescribe the normal component of the momentum:
\begin{align}\label{1-3}
    (\rho u_1)(L,x^{\prime}) = \bar\rho\bar u + \sigma m_{L}(x^{\prime}).
\end{align}
where \(m_{L} \in H^{\frac{5}{2}}(\Sigma)\) satisfying the compatibility conditions:
\begin{align}\label{compm}
     \int_{\Sigma}(m_{0}(x^{\prime}) - m_{L}(x^{\prime})) dx^{\prime} = 0.
\end{align}
\par On the cylinder wall \( \Gamma_{\omega}\), the usual slip boundary condition is imposed:
\begin{align}\label{1-5}
    n_2u_2 + n_3u_3 = 0,
\end{align}
where $(0,n_2,n_3)$ is the unit outer normal to  \( \Gamma_{\omega}\).

\par The main result in this paper is the following.

\begin{theorem}Given $(B_{0},K_{0}) \in (H_0^3(\Sigma))^2 $, $J_{0} \in H^2_0(\Sigma)$ and $(m_{0},m_{L}) \in (H^{\frac{5}{2}}(\Sigma))^2$ satisfying the compatibility conditions \eqref{compm}, there exists a positive small constant $\sigma_0$ depending on the background subsonic state $(\bar\rho,\bar u,0,0,\bar K)$ and $(m_{0},B_{0},K_{0},J_{0},m_{L})$, such that if $0 <  \sigma< \sigma_0$, then there exists a unique subsonic flow $(\rho, u_1,u_2,u_3,K) \in H^3(\Omega)$ to \eqref{euler} satisfying boundary conditions \eqref{1-2} and \eqref{1-3}. Moreover, the following estimate holds:
\begin{equation}
    \|(\rho,u_1,u_2,u_3,K) - (\bar\rho,\bar u,0,0,\bar K)\|_{3,\Omega} \leq \mc \sigma,
\end{equation}
where $\mc>0$ depends on $(\bar\rho, \bar u,\bar K)$ and $(m_{0},B_{0},K_{0},J_{0},m_{L})$.

\end{theorem}

\begin{remark}
The previous works  \cite{CS08,CY08,LXY16}  reduced the Euler system into a second-order elliptic equation for the pressure coupled with four transport equations of the velocity, the Bernoulli's quantity  and the entropy. Different from \cite{CS08,CY08,LXY16}, we use the deformation-curl decomposition introduced
in \cite{WX19, WS19} to decouple the hyperbolic and elliptic modes in the steady Euler system. The steady Euler system are rewritten as two transport equations for the Bernoulli's quantity and the entropy, two algebraic equations for two components of the vorticity, and a deformation-curl system for the velocity field. One of the crucial advantages for the deformation-curl decomposition is that the solutions obtained have the same regularity for the velocity, pressure and entropy.
\end{remark}

\begin{remark}
It should be noted that all the domains studied in  \cite{CX14,CY08, LXY16, WS15,  WX19, WS19} are rectangular cylinders or spherical shell. The reason for choosing these domains is to avoid the corner singularity near the intersections of the entrance, the exit and the nozzle wall. In \cite{SS16}, the author proved the existence and uniqueness of steady incompressible flows in the Sobolev space on cylinders with arbitrary cross section. The author in \cite{SS16} utilized the vector potential formulation of the div-curl system with normal boundary condition and the separation of variables to improve the regularity of the solution to the div-curl system near the intersections of the entrance, the exit and the nozzle wall. Here, we combine the method developed in \cite{SS16} and the deformation-curl decomposition in \cite{WX19, WS19} to establish the structural stability of the steady compressible Euler system in cylinders with arbitrary cross section.
\end{remark}

\par This paper will be organized as follows. In Section 2, we recall the deformation-curl decomposition to the steady Euler equations and present some results on the Sobolev spaces defined on a cylinder. In Section 3, we prove the structural stability of one dimensional subsonic background state under multi-dimensional perturbations on the entrance and exit of the cylinder.

\section{Some preliminaries}
\subsection{A deformation-curl reformulation to the steady Euler system}\noindent

\par The three dimensional steady Euler system is hyperbolic-elliptic coupled in subsonic regions,  whose effective decomposition into elliptic and hyperbolic modes is crucial for  developing a well-defined iteration.  Here we will use the deformation-curl decomposition to the steady Euler
system introduced \cite{WX19,WS19} to effectively decouple the hyperbolic and elliptic modes in subsonic region.

\par First, the Bernoulli's quantity $B = \frac{1}{2}|\textbf{u}|^2 + \frac{\gamma}{\gamma-1}K(S)\rho^{\gamma-1}$ and the entropy $K(S)$ satisfy the transport equations
\begin{equation}\label{B}
    (\textbf{u}\cdot \nabla)B = (\textbf{u}\cdot\nabla)K(S) = 0.
\end{equation}
Define the vorticity $\omega = \text{curl} \mathbf{u}=\sum_{i=1}^3\om {\bf e}_i$:
\begin{equation*}
    \omega_1 = \p_2u_3 - \p_3u_2,\; \omega_2 = \p_3u_1 - \p_1u_3,\; \omega_3 = \p_1u_2 - \p_2u_1.
\end{equation*}
From the third and fourth equation in \eqref{euler}, one can easily derive
\begin{equation}\label{w2}
    \omega_2(x_1,x^\prime) = \frac{u_2\omega_1 + \p_3 B -\frac{1}{\gamma-1}\rho^{\gamma-1}\p_3 K}{u_1},
 \end{equation}
    \begin{equation}\label{w3}
    \omega_3(x_1,x^\prime) = \frac{u_3\omega_1 - \p_2 B +\frac{1}{\gamma-1}\rho^{\gamma-1}\p_2 K}{u_1}.
\end{equation}
\par Note that
\begin{equation}\label{uw}
    \text{div  curl} \textbf{u} =\p_1 \om_1 +\p_2 \om_2+\p_3 \om_3= 0.
\end{equation}
 Substituting \eqref{w2} and \eqref{w3} into the equation  \eqref{uw} yields that
\begin{equation}\label{w1}
\begin{aligned}
    & \p_1\omega_1 + \sum_{i=2}^3\frac{u_i}{u_1}\p_i\omega_1 + \sum_{i=2}^3\p_i\bigg(\frac{u_i}{u_1} \bigg)\omega_1+ \p_2\bigg(\frac{1}{u_1} \bigg)\p_3 B -\p_3\bigg(\frac{1}{u_1} \bigg)\p_2B \\
    & \quad - \frac{1}{\gamma-1}\p_2\bigg(\frac{\rho^{\gamma-1}}{u_1} \bigg)\p_3K+ \frac{1}{\gamma-1}\p_3\bigg(\frac{\rho^{\gamma-1}}{u_1} \bigg)\p_2K=0.
    \end{aligned}
\end{equation}
 \par Next, we are going to study the elliptic modes  in \eqref{euler}. By definition of the Bernoulli's function, one has
 \begin{equation}\label{rho}
     \rho = H(B,K,|\textbf{u}|^2) = \bigg( \frac{\gamma-1}{\gamma K}\bigg)^{\frac{1}{\gamma-1}} \bigg( B -\frac{1}{2}|\textbf{u}|^2\bigg)^{\frac{1}{\gamma-1}}.
 \end{equation}
 Substituting \eqref{rho} into the equation in \eqref{euler}, we obtain the deformation form of the density
 \begin{equation*}
     H(B,K,|\textbf{u}|^2) \text{div}\textbf{u} + 2 \frac{\p H}{\p |\textbf{u}|^2}(B,K,|\textbf{u}|^2)u_iu_j\p_i u_j =0,
 \end{equation*}
which can be rewritten as a Frobenius inner product of a symmetric matrix and the deformation matrix
 \begin{equation}\label{rhorho}
     \sum_{i,j=1}^{3}\bigg(\delta_{ij} - \frac{u_iu_j}{c^2(\rho,K)} \bigg)D_{ij}(\textbf{u}) =0,
 \end{equation}
where $D_{ij}(\textbf{u})=\frac12 (\p_i u_j+\p_j u_i)$.

Combining \eqref{rhorho} with the curl system, we obtain the velocity field $\textbf{u}$ by solving the following  deformation-curl system:
\begin{align}\label{u}
    \begin{cases}
        \sum_{i,j=1}^{3}\big(\delta_{ij} - \frac{u_iu_j}{c^2(\rho,K)} \big)d_{ij}(\textbf{u}) =0,\\
        \text{curl }{\bf u}=(\om_1,\om_2,\om_3)^t.
    \end{cases}
\end{align}
After solving the above deformation-curl system to get the velocity, we use the formulation \eqref{rho} to recover the density and also the pressure.\\[1ex]
\par Then we have the following equivalence lemma.
\begin{lemma}(\textbf{Equivalence}).
Assume that  vector functions $(\rho,\mathbf{u},K)\in H^3(\Omega)$ defined on a domain $\Omega$ do not contain the vacuum (i.e. $\rho > 0 \;\text{in}\;\Omega$) and the horizontal velocity  is positive (i.e. $u_1 > 0 \;\text{in}\;\Omega$) , then the following two statements are equivalent:
\begin{enumerate}[\rm(i)]
 \item $(\rho,\mathbf{u},K)$ satisfy the steady Euler system \eqref{euler} in $\Omega$.
 \item $(\mathbf{u},B,K)$ satisfy the equations \eqref{B}, \eqref{w2}, \eqref{w3} and \eqref{u} in $\Omega$.
     \end{enumerate}
\end{lemma}

\subsection{Some results for  Sobolev spaces on $\Omega=(0,L)\times\Sigma$}\noindent

\par To solve the problem \eqref{euler}  with the boundary conditions \eqref{1-2} -\eqref{1-5} in  $ H^3(\Omega) $, we first give the embedding theorem and trace theorem for the Sobolev space in the cylinder $ \Omega $.

\begin{lemma}(Sobolev's embedding theorem \cite[Theorem \rm{5.4}]{RA75}).
\begin{enumerate}[ \rm (1)]
    \item Let $mp > 3$, then
\begin{equation*}
    W^{j+m,p}(\Omega) \hookrightarrow C_b^j(\Omega),
\end{equation*}
where
\begin{equation*}
    C_b^j(\Omega) = \{f \in C^j(\Omega): D^{\alpha}f \text{ is bounded for } |\alpha| \leq j\},
\end{equation*}
with norm
\begin{equation*}
    \|f\|_{C_b^j(\Omega)} = \max_{|\alpha|\leq j} \sup_{x \in \Omega}|f(x)|.
\end{equation*}
  Moreover, for all $f \in W^{j+m,p}(\Omega)$, there exists a constant $\mc_0 =\mc_0(\Omega) $ such that
\begin{equation}\nonumber
    \|f\|_{C_b^j(\Omega)} \leq \mc_0\|f\|_{W^{j+m,p}(\Omega)}.
\end{equation}

\item Let $1\leq p \leq \infty$, then
\begin{equation*}
    W^{1,p}([0,L];X) \hookrightarrow C([0,L];X),
\end{equation*}
which means that any function in $W^{1,p}([0,L];X)$ is a function in $C([0,L];X)$ and that there exists a constant $\mc_1$ such that
\begin{equation}\nonumber
    \|f\|_{C([0,L];X)} \leq \mc_1\|f\|_{W^{1,p}([0,L];X)},
\end{equation}
for all $f \in W^{1,p}([0,L];X)$.
\end{enumerate}
\end{lemma}

\begin{lemma}(Trace theorem \cite[Theorem \rm{7.53}]{RA75} and \cite{DZ96}).
 Let
    \begin{equation*}
        \hat{H}^s(\p\Omega) = H^s(\Gamma_0) \times H^s(\Gamma_L) \times H^s(\Gamma_w),
    \end{equation*}
    with compatibility conditions on $\p\Gamma_0 \cap \p\Gamma_w$ and  $\p\Gamma_L \cap \p\Gamma_w$ (depending on s). Then  for $f\in H^m(\Omega) \cap C(\bar{\Omega})$, there exists a bounded linear operator $T: H^m(\Omega) \to H^{m-\frac{1}{2}}(\p\Omega)$ such that
\begin{align}\nonumber
    Tf = f|_{\p \Omega},   \quad
    \|Tf\|_{H^{m-\frac{1}{2}}{(\p\Omega)}}\leq \mc_2 \|f\|_{H^m(\Omega)}.
\end{align}
\end{lemma}

Next, we  give some properties for the Sobolev space.

\begin{lemma}\label{theoremfg}
    Let $f \in H^p(\Omega)$  and $g \in H^q(\Omega)$, and let $ r \leq min\{p,q\}$ and $ r < p + q - \frac{3}{2} $. Then $fg \in H^r(\Omega)$ and there exists a constant $\mc_3$ such that
\begin{equation}\nonumber
    \|fg\|_{r,\Omega} \leq {\mc}_3\|f\|_{p,\Omega} \|g\|_{q,\Omega}.
\end{equation}
\end{lemma}

\begin{lemma}\label{theoremH3}
    For the Sobolev's space $H^2(\Om)$  and $H^3(\Omega)$ , one has
\begin{equation}\nonumber
    H^2(\Omega) = L^2([0,L];H^2(\Sigma)) \cap H^1([0,L];H^1(\Sigma)) \cap H^2([0,L];L^2(\Sigma)).
\end{equation}
and
\begin{equation}\nonumber
    H^3(\Omega) = L^2([0,L];H^3(\Sigma)) \cap H^1([0,L];H^2(\Sigma)) \cap H^2([0,L];H^1(\Sigma)) \cap  H^3([0,L];L^2(\Sigma)).
\end{equation}
\end{lemma}

\begin{definition}
    The space $C([0,L];X)$ consists of all continuous functions $f: [0,L] \to X$ and is equipped with the norm $\|f\|_{C([0,L];X)} = \max_{t \in [0,L]}\|f(t)\|_{X}$.
    \par The space $ C_w([0,L];X)$ consists of all weakly continuous functions $f: [0,L] \to X$, that is, functions $f$ such that $t \to l(f(t))$ is continuous for any bounded linear functional $l \in X^{\prime}$.
\end{definition}

\begin{definition}
If $h$ is defined on $\Sigma$, let
\begin{equation*}
    \text{ext}_0[h](x) =
    \begin{cases}
        h(x), \quad &{\rm{if}} \quad x \in \Sigma,\\
        0,\quad &\text{if} \quad x \notin \Sigma.
    \end{cases}
\end{equation*}
This allows us to define the space
\begin{equation*}
    H^{1/2}_{00}(\Sigma) = \{ h\in H^{1/2}(\Sigma):\text{ext}_0[h] \in H^{1/2}(\mbR^2) \}.
\end{equation*}
\end{definition}

\section{Existence and uniqueness of steady subsonic  flows}\noindent

\par In this section,  we will use the deformation-curl decomposition introduced in Section 2 to construct  a  subsonic flow which belongs to $ H^3(\Omega) $.

\subsection{ Proof of Theorem 1.1}\noindent
\par Define
\begin{align*}
     V_1 = u_1 -\bar u, \quad V_i = u_i, i = 2,3, \quad V_4 = B-\bar B,\quad V_5 = K- \bar K.
\end{align*}
We develop an iteration
as follows. Denote the solution space by
\begin{align}\nonumber
 \Xi_\delta = \bigg\{ \textbf{V} =(V_1,V_2,V_3,V_4,V_5):\sum_{i = 1}^{5} \| V_i\| _{3,\Omega} \leq \delta, \ V_2 n_2 + V_3 n_3=0\ \ \text{on }\Gamma_w\bigg\}.
 \end{align}

 Choose  any $ \bar{\textbf{V}} \in \Xi_\delta$, we will construct a new $\textbf{V}$ by the following three steps.

\par{\bf Step 1. Resolving the Bernoulli's quantity and the entropy.} In this step, we solve the transport equations
\begin{align}\label{3-7}
    \begin{cases}
        \p_1V_4 + \frac{\bar{V}_2}{\bar u + \bar{V}_1}\p_2V_4 + \frac{\bar{V}_3}{\bar u + \bar{V}_1}\p_3V_4 = 0,\\
        V_4(0,x^{\prime}) = \sigma B_{0}(x^{\prime}),
    \end{cases}
\end{align}
and
\begin{align}\label{3-8}
    \begin{cases}
         \p_1V_5 + \frac{\bar{V}_2}{\bar u + \bar{V}_1}\p_2V_5 + \frac{\bar{V}_3}{\bar u + \bar{V}_1}\p_3V_5 = 0,\\
        V_5(0,x^{\prime} )= \sigma K_{0}(x^{\prime}).
    \end{cases}
\end{align}
   To solve \eqref{3-7} and \eqref{3-8}, we first consider  the following transport eqaution:
\begin{align}\label{3-9}
    \begin{cases}
        \p_1f + (w_2,w_3) \cdot \nabla_{x^{\prime}} f = 0,\\
        f|_{x_1 = 0} = \eta(x^{\prime}).
    \end{cases}
\end{align}
\par For the equation \eqref{3-9}, we have the following conclusion:
\begin{lemma}\label{proposition3-1}
    Assume that $\mathbf{w}=(w_2,w_3) \in C^{\infty}(\ol{\Omega})$ satisfying $w_2n_2 + w_3 n_3=0$ on $\Ga_w$, and $f \in C^{\infty}(\Omega)$ with $f|_{x_1 =\xi} \in C^{\infty}_0(\Sigma)$ for all $\xi \in [0,L]$, then the solution $ f $ of \eqref{3-9}  satisfies
\begin{equation}\label{3-10}
\|f\|_{L^{\infty}([0,L];H^3(\Sigma))} +\|f\|_{L^2([0,L];H^3(\Sigma))}\leq C\|\eta\|_{3,\Sigma},
  \end{equation}
    where $C = C(L,\Sigma,\|\mathbf{w}\|_{3,\Omega})$.
\end{lemma}
\begin{proof}
First,
to get \eqref{3-10}, we need  the following inequality:
\begin{equation}\label{3-12}
\frac{d}{d x_1}\|f(x_1,\cdot)\|_{3,\Sigma}^2 \leq C(\Sigma) \| \mathbf{w}(x_1,\cdot)\|_{3,\Sigma}\|f(x_1,\cdot)\|_{3,\Sigma}^2.
\end{equation}
      Once \eqref{3-12} is obtained, then applying the Gronwall's inequality, one has
\begin{equation}\nonumber
    \sup_{x_1 \in [0,L]}\|f\|_{3,\Sigma}^2 \leq C(L,\Sigma,\| \mathbf{w}\|_{3,\Omega}) \| \eta\|_{3,\Sigma}^2.
\end{equation}
which proves $f \in L^{\infty}([0,L];H^3(\Sigma)).$ Thus
\begin{align} \nonumber
    \|f\|_{L^2([0,L];H^3(\Sigma))}^2 &=  \int_0^L \|f\|_{3,\Sigma}^2 d x_1
      \leq \int_0^L \sup_{x_1 \in [0,L]}\|f\|_{3,\Sigma}^2 d x_1\\\nonumber
     & \leq  C(L,\Sigma,\|\mathbf{w}\|_{3,\Omega}) \int_0^L  \| \eta\|_{3,\Sigma}^2 d x_1
     \leq  C(L,\Sigma,\|\mathbf{w}\|_{3,\Omega}) \| \eta\|_{3,\Sigma}^2,
\end{align}
which implies \( f \in L^2([0,L];H^3(\Sigma))\).
     \par Next, we prove \eqref{3-12}.
\begin{enumerate}[ \rm (1)]
\item Note that
     \begin{equation*}
    \frac{d}{d x_1}\int_{\Sigma}|f|^2 d x^{\prime} = -2 \int_{\Sigma} f\cdot (\mathbf{w}\cdot \nabla_{x^{\prime}} ) f d x^{\prime} =  \int_{\Sigma} f^2 \nabla_{x^{\prime}}\cdot \mathbf{w} d x^{\prime},
    \end{equation*}
where the boundary term disappears since \( f|_{x_1 =\xi } \in C_0^{\infty}(\Sigma)\). Hence one has
\begin{equation}\nonumber
    \frac{d}{dx_1}\|f(x_1,\cdot)\|_{0,\Sigma}^2 \leq \|\nabla_{x^{\prime}}\cdot \mathbf{w}(x_1,\cdot)\|_{C_b(\Sigma)}\|f(x_1,\cdot)\|_{0,\Sigma}^2 \leq C(\Sigma) \| \mathbf{w}(x_1,\cdot)\|_{3,\Sigma}\|f(x_1,\cdot)\|_{0,\Sigma}^2.
\end{equation}
   \item  Differentiating equation \eqref{3-9} with respect to \(x_i, i=2,3\) gives
   \begin{equation}\label{3-14}
    \p_1\p_if+\mathbf{w} \cdot \nabla_{x^{\prime}}\left(\p_if\right) + \p_i\mathbf{w}\cdot\nabla_{x^{\prime}} f = 0,\quad i = \text{2, 3}.
   \end{equation}
   A direct computation yields that
   \begin{equation*}
   \begin{aligned}
      &\frac{d}{d x_1}\int_{\Sigma}|\nabla_{x^{\prime}}f(x_1,\cdot)|^2 d x^{\prime} \\
    &\leq C(\Sigma)\|\nabla_{x^{\prime}}\cdot \mathbf{w}\|_{C_b(\Sigma)}\|\nabla_{x^{\prime}}f(x_1,\cdot)\|_{0,\Sigma} + C\|\nabla_{x^{\prime}}\mathbf{w}(x_1,\cdot)\|_{C_b(\Sigma)}\|\nabla_{x^{\prime}}f(x_1,\cdot)\|_{0,\Sigma}^2\\[2ex]
    &\leq   C(\Sigma) \| \mathbf{w}(x_1,\cdot)\|_{3,\Sigma}\|f(x_1,\cdot)\|_{1,\Sigma}^2.
   \end{aligned}
   \end{equation*}
     \item Differentiating equation \eqref{3-14} with respect to \(x_j, j=2,3\) yields
   \begin{equation}\label{3-15}
    \p_1\p^2_{ij}f+\mathbf{w} \cdot \nabla_{x^{\prime}}\left(\p^2_{ij}f\right) + \p_i\mathbf{w}\cdot\nabla_{x^{\prime}}\p_jf + \p_j\mathbf{w}\cdot\nabla_{x^{\prime}}\p_if+\p_{ij}^2\mathbf{w}\cdot\nabla_{x^{\prime}} f = 0.
    \end{equation}
    Then one gets
    \begin{equation*}
   \begin{aligned}
        &\frac{d}{dx_1}\|\nabla_{x^\prime}^2 f(x_1,\cdot)\|_{0,\Sigma}^2\\
       &\leq  \|\nabla_{x^{'}}\cdot \mathbf{w}\|_{C_b(\Sigma)}\|\nabla_{x^\prime}^2 f(x_1,\cdot)\|_{0,\Sigma}^2
       + \|\nabla_{x^\prime} \mathbf{w}(x_1,\cdot)\|_{C_b(\Sigma)}\| \nabla^2_{x^{\prime}}f\|_{0,\Sigma}\|\nabla_{x^\prime}^2 f\|_{0,\Sigma}\\[2ex]
        &\quad+\|\nabla^2_{x^{\prime}} \mathbf{w}(x_1,\cdot)\|_{1,\Sigma}\|\nabla_{x^{\prime}}f\|_{1,\Sigma}\|\nabla^2_{x^{\prime}} f(x_1,\cdot)\|_{0,\Sigma} \\[2ex]
        &\leq   C(\Sigma) \| \mathbf{w}(x_1,\cdot)\|_{3,\Sigma}\|f(x_1,\cdot)\|_{2,\Sigma}^2.
  \end{aligned}
   \end{equation*}
  \item Differentiating equation \eqref{3-15} with respect to \(x_k,k=2,3\) leads to
\begin{equation}\nonumber
\begin{aligned}
   & \p_1\left(\p_{ijk}^3 f\right)+\mathbf{w} \cdot \nabla_{x^{\prime}}\left(\p_{ijk}^3 f\right) + \p_i \mathbf{w} \cdot \nabla_{x^{\prime}} (\p^2_{jk}f)+\p_j \mathbf{w} \cdot \nabla_{x^{\prime}} (\p^2_{ik}f)+\p_k \mathbf{w} \cdot \nabla_{x^{\prime}} (\p^2_{ij}f)  \\
  & + \p_{ij}^2 \mathbf{w} \cdot \nabla_{x^{\prime}} \left( \p_k f\right)+\p_{ik}^2 \mathbf{w} \cdot \nabla_{x^{\prime}} \left( \p_j f\right)+ \p_{jk}^2 \mathbf{w} \cdot \nabla_{x^{\prime}} \left( \p_i f\right)+\p_{ijk}^3 \mathbf{w}\cdot\nabla_{x^{\prime}} f = 0.
\end{aligned}
\end{equation}
A direct calculation gives that
 \begin{equation*}
   \begin{aligned}
        &\frac{d}{dx_1}\|\nabla_{x^\prime}^3 f(x_1,\cdot)\|_{0,\Sigma}^2\\
       &\leq  \|\nabla_{x^{\prime}}\cdot \mathbf{w}\|_{C_b(\Sigma)}\|\nabla_{x^\prime}^3 f\|_{0,\Sigma}^2 + 3\|\nabla_{x^{\prime}}\mathbf{w}\|_{C_b(\Sigma)}\|\nabla_{x^\prime}^3 f\|_{0,\Sigma}\|\nabla_{x^\prime}^3 f\|_{0,\Sigma}\\[2ex]
     &\quad+3\|\p_2^2 \mathbf{w}\|_{1,\Sigma}\|\nabla_{x^\prime}^3 f \|_{1,\Sigma}\|\nabla_{x^\prime}^3 f\|_{0,\Sigma}    +\|\nabla_{x^{\prime}}f\|_{2,\Sigma}\| \mathbf{w} \| _{3,\Sigma}\|\nabla_{x^\prime}^3 f\|_{0,\Sigma}  \\[2ex]
          &\leq   C(\Sigma) \| \mathbf{w}(x_1,\cdot)\|_{3,\Sigma}\|f(x_1,\cdot)\|_{3,\Sigma}^2.
   \end{aligned}
   \end{equation*}
\end{enumerate}
   Therefore,  the inequality  \eqref{3-12} is proved.

\end{proof}
\begin{lemma}\label{lemma3-2}
    Assume that $\mathbf{w}=(w_2,w_3) \in C^{\infty}(\ol{\Omega})$ satisfying $w_2n_2 + w_3 n_3=0$ on $\Ga_w$, and $f \in C^{\infty}(\Omega)$ with $f|_{x_1 =\xi} \in C^{\infty}_0(\Sigma)$ for all $\xi \in [0,L]$, then the solution $ f $ of \eqref{3-9}  satisfies
\begin{equation}\label{3-h1}
\|f\|_{W^{1,\infty}([0,L];H^2(\Sigma))}+\|f\|_{H^1([0,L];H^2(\Sigma))}\leq C\|\eta\|_{3,\Sigma},
  \end{equation}
  and
    \begin{equation}\label{3-h2}
 \|f\|_{H^2([0,L];H^1(\Sigma))} \leq  C \|\eta\|_{3,\Sigma},
    \end{equation}
    and
    \begin{equation}\label{3-11}
 \|f\|_{H^3([0,L];L^2(\Sigma))} \leq  C \|\eta\|_{3,\Sigma},
    \end{equation}
    where $C = C(\Sigma,L,\|\mathbf{w}\|_{3,\Omega})$.
\end{lemma}
\begin{proof}
    It follows from the equation \eqref{3-9} that
    \begin{equation*}
    \begin{aligned}
           \|\p_1 f(x_1,\cdot)\|_{2,\Sigma} & = \|\mathbf{w}\cdot\nabla_{x^{\prime}}f(x_1,\cdot)\|_{2,\Sigma} \leq \|\mathbf{w}(x_1,\cdot)\|_{2,\Sigma} \|\nabla_{x^{\prime}}f(x_1,\cdot)\|_{2,\Sigma}\\
        & \leq \|\mathbf{w}(x_1,\cdot)\|_{2,\Sigma} \|f(x_1,\cdot)\|_{3,\Sigma},
    \end{aligned}
    \end{equation*}
     and  \eqref{3-h1} follows from this and \eqref{3-10}.
    \par Differentiating the equation \eqref{3-9} with respect to \( x_1\) gives
\begin{equation}\nonumber
    \p_1^2 f+\mathbf{w} \cdot \nabla_{x^{\prime}}\left( \p_1 f \right) +\p_1 \mathbf{w}\cdot\nabla_{x^{\prime}} f =0,
\end{equation}
  which implies that
   \begin{align*}
      \| \p_1^2 f\|_{1,\Sigma} &\leq \| \p_1 \mathbf{w}\cdot \nabla_{x^{'}}f\|_{1,\Sigma} + \| \mathbf{w} \cdot \nabla_{x^{'}}\p_1 f  \|_{1,\Sigma} \\[2ex]
      &\leq C(\Sigma)\left( \| \p_1 \mathbf{w}\|_{1,\Sigma}\| \nabla_{x^{\prime}}f \|_{2,\Sigma}  + \|\mathbf{w}\|_{2,\Sigma}\| \nabla_{x^{\prime}}\left(\p_1 f\right) \|_{1,\Sigma} \right)\\[2ex]
      & \leq C(\Sigma)\|\mathbf{w}\|_{2,\Sigma}\|f\|_{3,\Sigma}.
   \end{align*}
These imply \eqref{3-h2} .
\par Differentiating the equation \eqref{3-9}  with respect to \(x_1\) twice, one gets
   \begin{equation}\nonumber
    \p_1^3 f+\mathbf{w} \cdot \nabla_{x^{\prime}}\left(\p_1^2 f\right) +2\left( \p_1 \mathbf{w} \cdot \nabla_{x^{\prime}}\right)  \p_1 f+ \p_1^2 \mathbf{w}\cdot\nabla_{x^{\prime}} f = 0.
   \end{equation}
Then
        \begin{align*}
       \|\p_1^3 f \|_{0,\Sigma} & \leq \|\mathbf{w}\cdot \nabla_{x^{\prime}} \p_1^2 f\|_{0,\Sigma} +2\| \p_1\mathbf{w} \cdot \nabla_{x^{\prime}}(\p_1f) \|_{0,\Sigma} +\| \p_1^2 \mathbf{w} \cdot \nabla_{x^{\prime}}f \|_{0,\Sigma}\\[2ex]
       & \leq \|\mathbf{w}\|_{2,\Sigma}\|\nabla_{x^{\prime}}\p_1^2 f\|_{0,\Sigma} +\|\p_1\mathbf{w}\|_{2,\Sigma}\|\nabla_{x^{\prime}}\p_1 f\|_{0,\Sigma} + \|\p_1^2\mathbf{w}\|_{1,\Sigma}\|\nabla_{x^{\prime}}f\|_{1,\Sigma}\\[2ex]
        & \leq C(\Sigma)\|\mathbf{w}\|_{3,\Sigma}\|f\|_{3,\Sigma}\leq C\left(L,\Sigma,\|\mathbf{w}\|_{3,\Omega}\right) \|\mathbf{w}\|_{3,\Sigma}   \|\eta\|_{3,\Sigma},
   \end{align*}
   which implies
   \begin{align*}
       \int_0^L \|\p_1^3 f \|_{0,\Sigma}^2 dx_1 \leq \int_0^L C\left(L,\Sigma,\|\mathbf{w}\|_{3,\Omega}\right) \|\mathbf{w}\|_{3,\Sigma} ^2  \|\eta\|_{3,\Sigma}^2 dx_1\leq C\left(L,\Sigma,\|\mathbf{w}\|_{3,\Omega}\right) \|\mathbf{w}\|_{3,\Omega}^2   \|\eta\|_{3,\Sigma}^2.
 \end{align*}

 Therefore the proof of Lemma \ref{lemma3-2} is completed.
\end{proof}
\par Next, we establish the existence and uniqueness of the solution to the equation \eqref{3-9}.

\begin{proposition} \label{lemma3-5}For $\eta \in H^3_0(\Sigma)$ and $\mathbf{w}=(w_2,w_3)\in H^3(\Omega)$ satisfying $w_2n_2+w_3 n_3=0$ on $\Ga_w$, the equation \eqref{3-9} has a unique solution $ f \in H^3(\Omega)\cap L^{\infty}([0,L];H_0^3(\Sigma)) \cap W^{1,\infty}([0,L];H_0^2(\Sigma))$.
Moreover, for almost every $x_1 \in [0,L]$, the first equation in \eqref{3-9} holds as an equality in $H^2(\Sigma)$ and $f|_{x_1=0} = \eta$ in $H_0^2(\Sigma)$.
\end{proposition}
\begin{proof}
     First,  $\mathbf{w} $ and $\eta$ can be approximated by a sequence of functions $\{\eta_i\}_{i=1}^{\infty}\in  C_0^{\infty}(\Sigma)$ and $\{\mathbf{w}_i\}_{i=1}^{\infty}\in  C^{\infty}(\overline{\Omega})$ satisfying $w_{2i} n_2 + w_{3i} n_3=0$ on $\Ga_w$. Thus we use the standard characteristic method to solve these transport equations. Letting $(\tau,\bar{x}_2(\tau;x),\bar{x}_3(\tau;x))$ be the solution to the following ODE system
\begin{eqnarray}\label{char}
    \begin{cases}
       \frac{d \bar{x}^i_2(\tau;x)}{d\tau} = w_{2i}(\tau,\bar{x}_2^i(\tau;x),\bar{x}_3^i(\tau;x)),\\
       \frac{d \bar{x}_3^i(\tau;x)}{d\tau} = w_{3i}(\tau,\bar{x}_2^i(\tau;x),\bar{x}_3^i(\tau;x)),\\
       \bar{x}_2(x_1;x) = x_2,\;\bar{x}_3(x_1;x) = x_3.
    \end{cases}
\end{eqnarray}
Then one has $f_{i}(x) = \eta_i(\bar{x}_2^i(0;x),\bar{x}_3^i(0;x))$ and $f_i|_{x_1=\xi}\in  C_0^{\infty}(\Sigma)$ for any $\xi\in [0,L]$.

\par The sequence $\{f_{i}\}_{i=1}^{\infty}$ is bounded in $H^3(\Omega)$, thus one can extract a subsequence which converges weakly to a function $f\in H^3(\Omega)$. We further prove that $ f\in L^{\infty}([0,L];H_0^3(\Sigma)) \cap W^{1,\infty}([0,L];H_0^2(\Sigma))$.

Thanks to the estimate \eqref{3-10}, the sequence $\{f_{i}|_{x_1=\xi}\}_{i=1}^{\infty}$ in $H_0^3(\Sigma)$ for any $\xi \in [0,L]$, from which one can extract a weakly convergent subsequence with a limit $\hat{f}_{\xi} \in H_0^3(\Sigma)$.

For any linear functional $l$ on $H^2(\Sigma)$, define a new linear functional on $H^1([0,L];H^2(\Sigma))\supset C([0,L];H^2(\Sigma)$ as follows
$$l_{\xi}(g) = l(g|_{x_1 = \xi}).$$

Then
\begin{equation*}
    l(f|_{x_1 = \xi}) = l_{\xi}(f) = \lim_{i \to \infty} l_{\xi}(f_{i}) = \lim_{i \to \infty} l(f_{i}|_{x_1 = \xi}) = l(\hat{f}_{\xi})
\end{equation*}
for all linear functionals on $H^2(\Sigma)$ and hence $f|_{x_1 =\xi} = \hat{f}_{x_1}$ in $H^2(\Sigma)$. However, since $\hat{f}_{x_1} \in H_0^3(\Sigma)$ the equality also holds in $H_0^3(\Sigma)$. It follows that $f|_{x_1 = \xi} \in H_0^3(\Sigma)$ and hence $f \in  L^{\infty}([0,L];H_0^3(\Sigma)) $. By the same argument, one may get $ \p_1 f|_{x_1 = \xi} \in H_0^2(\Sigma)$, which together with $f|_{x_1 =\xi} \in H_0^3(\Sigma)$ gives $f \in W^{1,\infty}([0,L];H_0^2(\Sigma)$.

\par It remains to show that $f $ is a unique solution to \eqref{3-9}. It follows from the embedding $H^1([0,L]; H^2(\Sigma))\subset C([0,L];H^2(\Sigma))$ that
\begin{equation*}
    \mathbf{w}_{i}|_{x_1 = \xi} \rightarrow \mathbf{w}|_{x_1 = \xi} \quad\text{in} \quad H^2(\Sigma),
\end{equation*}
as $i\rightarrow \infty$. This, together with the fact
\begin{equation*}
    f_{i}|_{x_1 = \xi} \rightharpoonup f|_{x_1 = \xi} \quad\text{in} \quad H_0^3(\Sigma), \ \ i\to \infty
\end{equation*}
\begin{equation*}
   \p_1 f_{i}|_{x_1 = \xi} \rightharpoonup \p_1 f|_{x_1 = \xi} \quad\text{in} \quad H_0^2(\Sigma), \ \ i\to \infty
\end{equation*}
implies that
\begin{equation*}
    \left(\p_1 f_{i} + (\mathbf{w}_{i}\cdot \nabla) f_{i}\right)|_{x_1 = \xi} \rightharpoonup \left(\p_1 f + (\mathbf{w}\cdot \nabla) f\right)|_{x_1 = \xi} \quad \text{in} \quad H_0^2(\Sigma),
\end{equation*}
as $i \rightarrow \infty$. Note that for any $ i $,
\begin{equation*}
     \quad \|\p_1 f_{i} + (\mathbf{w}_{i}\cdot \nabla) f_{i}\|_{2,\Sigma} = 0.
\end{equation*}
Thus
\begin{equation*}
    \|\p_1 f + (\mathbf{w}\cdot \nabla) f\|_{2,\Sigma} \leq \lim\inf_{i \to \infty}\|\p_1 f_{i} + (\mathbf{w}_{i}\cdot \nabla) f_{i}\|_{2,\Sigma}=0.
\end{equation*}
Hence the first equation in \eqref{3-9} holds for $f$ as an equality in $H_0^2(\Sigma)$ for almost every $x_1 \in [0,L]$. That $f|_{x_1 = 0} = \eta$ follows immediately from
\begin{equation*}
    f_{i}|_{x_1 = 0} \rightharpoonup f|_{x_1 = 0} \quad \text{in} \quad H_0^3(\Sigma),
\end{equation*}
and
\begin{equation*}
    f_{i}|_{x_1 = 0} = \eta_i \to \eta \quad \text{in} \quad H_0^3(\Sigma).
\end{equation*}
\par The uniqueness of the solutions to \eqref{3-9} follows from a simple energy estimate.
\end{proof}
\par It follows from Lemmas \rm{3.1}, \rm{3.2}, \rm{3.3}  that \eqref{3-7} and \eqref{3-8} has a unique solution $ (V_4,V_5)\in H^3(\Omega) $ satisfying
\begin{equation}\label{3-a}
\|(V_4,V_5)\|_{3,\Omega}\leq C(L,\Sigma,\|(B_{0},K_{0})\|_{3,\Sigma})\sigma.\\[2ex]
\end{equation}

\par{\bf Step 2. Resolving the vorticity field.}
\par In this step, we  first solve the following linearized transport equation for \(\omega_1\).
\begin{align}\label{3-20}
\begin{cases}
    \p_1\omega_1 + \sum_{i=2}^{3}\frac{\bar{V}_i}{\bar u + \bar{V}_1} \p_i\om_1+ \mu(\bar{\mathbf{V}})\om_1 = R(\bar{\mathbf{V}},V_4,V_5),\quad in \quad \Omega,\\
    \om_1(0,x^{\prime}) = \sigma J_{0}(x^{\prime}).\quad on \quad x_1 = 0,
\end{cases}
\end{align}
where
\begin{align*}
    & \mu(\bar{\mathbf{V}})= \p_2\left(\frac{\bar{V}_2}{\bar u + \bar{V}_1}\right) + \p_3\left(\frac{\bar{V}_3}{\bar u + \bar{V}_1}\right),\\
    & R(\bar{\mathbf{V}},V_4,V_5) = -\p_2\left(\frac{1}{\bar u + \bar{V}_1}\right)\p_3V_4+\p_3\left(\frac{1}{\bar u + \bar{V}_1}\right)\p_2 V_4 \\
    &\qquad- \frac{1}{\gamma-1}\p_3\left(\frac{(H(\bar{\mathbf{V}}))^{\gamma-1}}{\bar u +\bar{V}_1 }\right)\p_2 V_5+\frac{1}{\gamma-1}\p_2\left(\frac{(H(\bar{\mathbf{V}}))^{\gamma-1}}{\bar u + \bar{V}_1}\right)\p_3 V_5,\\
    & (H(\bar{\mathbf{V}}))^{\gamma-1} = \left( \frac{\gamma-1}{\gamma(\bar K + \bar{V}_5)} \right)\left( B_0 + \bar{V}_4 -\frac{1}{2}((\bar u + \bar{V}_1)^2 + \bar{V}_2^2 + \bar{V}_3^2) \right),
\end{align*}
and from Step 1, we clearly know that $V_4$ and $V_5 \in L^{\infty}([0,L];H_0^3(\Sigma)) \cap H^3(\Omega)$, which implies $ R \in L^2([0,L];H_0^2(\Sigma))$.
\par To solve \eqref{3-20}, we consider  the following transport equation:
\begin{align}\label{3-21}
    \begin{cases}
        \p_1 \omega_1 + (w_2,w_3) \cdot \nabla_{x^{\prime}} \omega_1  +\mu(\bar{\mathbf{V}})\omega_1 = R(\bar{\mathbf{V}},V_4,V_5), \\
        \omega_1|_{x_1 = 0} =  \sigma J_{0}(x^{\prime}).
    \end{cases}
\end{align}
where $ w_i = \frac{\bar{V}_i}{\bar u + \bar{V}_1}$, $i =2,3$.
\par Similar to Lemmas \rm{3.1} and \rm{3.2}, we have the following conclusion.
\begin{lemma}\label{lemma3-4}
Assume that ${\bf w}=(w_2,w_3) \in C^{\infty}(\ol{\Omega})$ satisfying $w_2n_2+w_3 n_3=0$ on $\Ga_w$, $ \mu \in C^{\infty}(\ol{\Omega})$ and $\omega_1\in C^{\infty}(\Omega)  $ with $\omega_1|_{x_1 = \xi} \in C_0^{\infty}(\Sigma)$  for all $\xi \in [0,L]$, then the solution $ \omega_1 $ of \eqref{3-21}   satisfies
\begin{equation}\label{3-22}
       \|\omega_1\|_{L^{\infty}([0,L];H^2(\Sigma))}+ \|\omega_1\|_{L^2([0,L];H^2(\Sigma))} \leq C( \sigma\|J_{0}\|_{2,\Sigma} +  \|R\|_{2,\Omega}),
\end{equation}
and
\begin{equation}\label{3-233}
    \|\omega_1\|_{W^{1,\infty}([0,L];H^1(\Sigma))}+\|\omega_1\|_{H^1([0,L];H^1(\Sigma))} \leq C (\sigma\|  J_{0}\|_{2,\Sigma}+ \|R\|_{2,\Omega}),
\end{equation}
and
\begin{equation}\label{3-23}
    \|\omega_1\|_{H^2([0,L];L^2(\Sigma))} \leq C (\sigma \| J_{0}\|_{2,\Sigma}+\|R\|_{2,\Omega}),
\end{equation}

where $C=C(L,\Sigma,\|\mathbf{w}\|_{3,\Omega},\|\mu\|_{2,\Omega})$.
\end{lemma}

\begin{proof}
\begin{enumerate}
\item  To prove \eqref{3-22}, we first show the following inequality:
\begin{equation}\label{3-24}
    \frac{d}{dx_1}\|\omega_1\|_{2,\Sigma}^2 \leq  C ( \Sigma)(\|\mathbf{w}(x_1,\cdot)\|_{3,\Sigma} + \|\mu\|_{2,\Sigma} +1 )\|\omega_1(x_1,\cdot)\|_{0,\Sigma}^2 + \|R\|_{2,\Sigma}^2.
\end{equation}
Applying the Gronwall's inequality, one has
\begin{align*}
    \sup_{x_1 \in [0,L]}\|\omega_1\|_{2,\Sigma}^2
    &\leq \exp{\left(\int_0^L C(\Sigma)(\|\mathbf{w}\|_{3,\Sigma}+\|\mu\|_{2,\Sigma}+1) dx_1\right)}\left[
\sigma\|  J_{0}\|_{2,\Sigma}^2 + \int_0^L \|R\|_{2,\Sigma}^2  dx_1\right]\\
    & \leq C\left(L,\Sigma,\| \mathbf{w}\|_{3,\Omega},\|\mu\|_{2,\Omega}\right)\left(\sigma\|J_{0}\|_{2,\Sigma}^2 + \|R\|_{2,\Om}^2\right).
\end{align*}
Thus
\begin{align*}
    \|\omega_1\|_{L^2([0,L];H^2(\Sigma))}^2= \int_0^L\|\omega_1\|_{2,\Sigma}^2d x_1 \leq C\left(L,\Sigma,\| \mathbf{w}\|_{3,\Omega},\|\mu\|_{2,\Omega}\right) \left( \sigma\|J_{0}\|_{2,\Sigma}^2 + \|R\|_{2,\Om}^2\right).
\end{align*}
\par Next, we prove \eqref{3-24}. Note that
\begin{align*}
    \frac{d}{dx_1}\|\omega_1(x_1,\cdot)\|_{0,\Sigma}^2 &=2\int_{\Sigma}R\omega_1 d x^{\prime}+\int_{\Sigma}\omega_1^2 \nabla_{x^{'}}\cdot\mathbf{w} d x^{\prime} -2\int_{\Sigma}\mu \omega_1^2d x^{\prime}\\
    & \leq \left(\| \nabla_{x^{\prime}} \cdot \mathbf{w} \|_{C_b(\Sigma)} + 2\|\mu\|_{C_b(\Sigma)} +1\right)\|\omega_1\|_{0,\Sigma}^2 + \|R\|_{0,\Sigma}^2.
\end{align*}
Differentiating the equation \eqref{3-21} with respect to \(x_i,i = 2,3\) gives:
    \begin{align*}
    \p_1\p_i \omega_1+\mathbf{w} \cdot \nabla_{x^{'}}\left(\p_i \omega_1\right) + \p_i \mathbf{w}\cdot\nabla_{x^{'}} \omega_1 + \omega_1\p_i \mu + \mu \p_i \omega_1 = \p_i R, \quad i = 2,3.
   \end{align*}
A direct computation yields that
   \begin{align*}
      &\frac{d}{d x_1}\|\nabla_{x^\prime} \omega_1(x_1,\cdot)\|_{0,\Sigma}^2\\[2ex]
&\leq C(\Sigma)\|\nabla_{x^{'}}\cdot \mathbf{w}\|_{C_b(\Sigma)}\|\nabla_{x^\prime} \omega_1\|_{0,\Sigma}^2 + \|\nabla_{x^\prime} \mathbf{w}\|_{C_b(\Sigma)}\|\nabla_{x^{\prime}}\omega_1\|_{0,\Sigma}^2+C(\Sigma)\|\nabla_{x^\prime} R\|_{0,\Sigma}^2 \\[2ex] \nonumber
    & +C(\Sigma)\|\nabla_{x^\prime} \omega_1\|_{0,\Sigma}^2+
    C(\Sigma)\|\nabla_{x^\prime} \mu\|_{1,\Sigma}\|\nabla_{x^\prime} \omega_1\|_{0,\Sigma}\|\omega_1\|_{1,\Sigma} +C(\Sigma)\|\mu\|_{C_b(\Sigma)}\|\nabla_{x^\prime} \omega_1\|_{0,\Sigma}^2 \\[2ex] \nonumber
    & \leq C ( \Sigma)(\|\mathbf{w}(x_1,\cdot)\|_{3,\Sigma} + \|\mu\|_{2,\Sigma} +1 )\|\omega_1(x_1,\cdot)\|_{1,\Sigma}^2 + C(\Sigma)\|R\|_{1,\Sigma}^2 .
\end{align*}
Differentiating the equation \eqref{3-21}  with respect to \(x_i, x_j\) for $i,j=2,3$ twice yields
    \begin{eqnarray}\no
    &&\p_1(\p_{ij}^2 \omega_1)+\mathbf{w} \cdot \nabla_{x^\prime}\left(\p_{ij}^2 \omega_1\right) +\p_i \mathbf{w} \cdot \nabla_{x^\prime} \p_j \omega_1+\p_j \mathbf{w} \cdot \nabla_{x^\prime} \p_i \omega_1+ \p_{ij}^2 \mathbf{w}\cdot\nabla_{x^\prime} \omega_1\\\no
    &&\quad\quad +\p_{ij}^2 \mu\omega_1 +  \p_i\mu\p_j \omega_1 + \p_j \mu\p_i \omega_1+ \mu \p_{ij}^2 \omega_1 = \p_{ij}^2 R.
   \end{eqnarray}
Then one derives
     \begin{align*}
       \frac{d}{dx_1}\|\nabla_{x^\prime}^2 \omega_1\|_{0,\Sigma}^2 &\leq \left(\|\nabla_{x^{'}}\cdot \mathbf{w}\|_{C_b(\Sigma)} + 2\|\mu\|_{C_b(\Sigma)}\right)\|\nabla_{x^\prime}^2 \omega_1\|_{0,\Sigma}^2\\
       & + C(\Sigma)(\|\nabla_{x^\prime} \mathbf{w}\|_{C_b(\Sigma)}\|\nabla_{x^\prime}^2 \omega_1\|_{0,\Sigma}^2 + \|\nabla_{x^\prime}^2 \mathbf{w}\|_{1,\Sigma} \|\nabla_{x^{\prime}}\omega_1\|_{1,\Sigma}\|\nabla_{x^\prime}^2 \omega_1\|_{0,\Sigma}) \\[2ex]
        & + C(\Sigma)(\|\nabla_{x^\prime} \mu\|_{1,\Sigma}\|\nabla_{x^\prime} \omega_1\|_{1,\Sigma}\|\nabla_{x^\prime}^2 \omega_1\|_{0,\Sigma}+ \|\nabla_{x^\prime}^2 \mu\|_{0,\Sigma}\|\nabla_{x^\prime}^2 \omega_1\|_{0,\Sigma}\|\omega_1\|_{C_b(\Sigma)})\\
       & + C(\Sigma)(\int_{\Sigma} |\nabla_{x^\prime}^2 R|^2 d x' +\int_{\Sigma} |\nabla_{x'}^2 \omega_1|^2 d x')\\
       & \leq  C ( \Sigma)(\|\mathbf{w}(x_1,\cdot)\|_{3,\Sigma} + \|\mu\|_{2,\Sigma} +1 )\|\omega_1\|_{2,\Sigma}^2 + C(\Sigma)\|R\|_{2,\Sigma}^2.
   \end{align*}
Therefore, we completed the proof of inequality \eqref{3-24}.
\item  Next, to prove \eqref{3-233}, by using the equation \eqref{3-21} we get the estimate
\begin{equation*}
    \begin{aligned}
        \|\p_1 \omega_1\|_{1,\Sigma} &\leq \|R\|_{1,\Sigma} + \|\mathbf{w}\cdot\nabla_{x^{\prime}}\omega_1\|_{1,\Sigma} +\|\mu \omega_1\|_{1,\Sigma}\\[2ex]
        & \leq \|R\|_{1,\Sigma} + C(\|\mathbf{w}\|_{2,\Sigma} + \|\mu \|_{1,\Sigma})\|\omega_1\|_{2,\Sigma}
    \end{aligned}
\end{equation*}
This implies
\begin{equation*}
    \sup_{x_1 \in [0,L]} \|\p_1 \omega_1\|_{1,\Sigma} \leq   C(\|\mathbf{w}\|_{3,\Om} + \|\mu \|_{2,\Omega}+1)(\sigma\|  J_{0}\|_{2,\Sigma} +\|R\|_{2,\Om}),
\end{equation*}
by using $H^n(\Om) \subset H^1([0,L];H^{n-1}(\Sigma))$ and Lemma \rm{2.2}. This allows us to estimate
\begin{eqnarray}\no
    &&\int_0^L \|\p_1 \omega_1\|_{1,\Sigma}^2 dx_1 \leq C(\sigma\|  J_{0}\|_{2,\Sigma}^2 +\|R\|_{2,\Om}^2),\\\no
    &&\|\omega_1\|_{H^1([0,L];H^1(\Sigma))}^2= \int_0^L \|\omega_1\|_{1,\Sigma}^2+\|\p_1 \omega_1\|_{1,\Sigma}^2 dx_1 \\\no
    &&\leq C(L,\Sigma,\|\mathbf{w}\|_{3,\Omega},\|\mu\|_{2,\Omega})\sigma\|  J_{0}\|_{2,\Sigma}^2 + C\|R\|_{2,\Om}^2,
\end{eqnarray}
by using equation \eqref{3-22}. The inequality \eqref{3-233} is proved.

  \item  Next, we prove \eqref{3-23}. Differentiating the equation \eqref{3-21} with respect to \( x_1\) gives
  \begin{align*}
    \p_1^2 \omega_1+\mathbf{w}\cdot \nabla_{x^{'}}\left(\p_1 \omega_1\right) + \p_1 \mathbf{w}\cdot\nabla_{x^{'}} \omega_1 + \omega_1\p_1 \mu + \mu \p_1 \omega_1 = \p_1 R,
  \end{align*}
   which yields that
   \begin{align*}
      &\| \p_1^2 \omega_1\|_{0,\Sigma}\leq C(\Sigma)\bigg(\|\p_1 \mathbf{w}\|_{1,\Sigma}\|\nabla_{x^{\prime}}\omega_1\|_{1,\Sigma} + \|\mathbf{w}\|_{2,\Sigma}\|\nabla_{x^{\prime}} \p_1 \omega_1\|_{0,\Sigma}
       \\
       &\quad\quad+ \| \omega_1\|_{2,\Sigma}\| \p_1 \mu \|_{0,\Sigma} + \| \mu \|_{1,\Sigma}\| \p_1 \omega_1\|_{1,\Sigma} + \| \p_1 R\|_{0,\Sigma}\bigg)\\[2ex]
      & \leq C(\Sigma)\left(\| \p_1 \mathbf{w}\|_{1,\Sigma} + \| \p_1 \mu \|_{1,\Sigma}+\| \mathbf{w} \|_{2,\Sigma} +\| \mu \|_{1,\Sigma}\right)(\|\omega_1\|_{2,\Sigma}+\|\p_1 \omega_1\|_{1,\Sigma}) + C(\Sigma)\|\p_1 R\|_{0,\Sigma}.
   \end{align*}
  From this we get
    \begin{align*}
        \int_0^L \| \p_1^2 \omega_1\|^2_{0,\Sigma} dx_1 & \leq C\int_0^L \left(\sigma\|  J_{0}\|_{2,\Sigma}+  \|R\|_{1,\Sigma}\right)^2 d x_1+C\int_0^L\|\p_1R\|_{0,\Si}^2 dx_1\\
        &\leq C(\sigma\| J_{0} \|^2_{2,\Sigma} +\| R \|^2_{1,\Omega})
    \end{align*}
     where $C= C(L,\Sigma,\| \mathbf{w}\|_{3,\Omega},\|\mu\|_{2,\Omega})$.
  Then
  \begin{equation*}
  \begin{aligned}
            \|\omega_1\|^2_{H^2([0,L];L^2(\Sigma))}&= \int_0^L \|\omega_1\|^2_{0,\Sigma} + \|\p_1 \omega_1\|^2_{0,\Sigma} +\|\p_1^2 \omega_1\|^2_{0,\Sigma} dx_1\\
        & \leq C(L,\Sigma,\|\mathbf{w}\|_{3,\Omega},\|\mu\|_{2,\Omega})(\sigma\|  J_{0}\|_{2,\Sigma}^2 + \|R\|_{2,\Om}^2),
  \end{aligned}
  \end{equation*}
  by using \eqref{3-22} and \eqref{3-233}. Thus the proof of Lemma \ref{lemma3-4} is completed.

\end{enumerate}
\end{proof}

\par Next, we establish the existence and uniqueness of the solution to the equation \eqref{3-21}.

\begin{proposition}\label{proposition35}
 For $\mathbf{w}=(w_2,w_3)\in H^3(\Omega)$ satisfying $w_2n_2+w_3 n_3=0$ on $\Ga_w$, $J_{0} \in H^2_0(\Sigma)$ and $ R \in L^2([0,L];H_0^2(\Sigma))$, the equation \eqref{3-21} has a unique solution $ \omega_1 \in H^2(\Omega)\cap L^{\infty}([0,L];H_0^2(\Sigma)) \cap W^{1,\infty}([0,L];H^1_0(\Sigma))$. Moreover, for almost every $x_1 \in [0,L]$, the first equation in \eqref{3-21} holds as an equality in $H_0^1(\Sigma)$ and $\omega_1|_{x_1=0} =  \sigma J_{0}$ in $H_0^2(\Sigma)$. Furthermore, there holds
\begin{equation*}\label{351}
    \p_1 \omega_1|_{x_1 = 0,L}\in H_0^1(\Sigma) \subset H^{1/2}_{00}(\Sigma).
\end{equation*}

\end{proposition}
\begin{proof}

Same as the proof of Proposition \ref{lemma3-5},  $\mathbf{w} $, $J_0$ and $R$ can be approximated by a sequence of functions $\{\mathbf{w}_i\}_{i=1}^{\infty}\in  C^{\infty}(\overline{\Omega})$, $\{J_0^i\}_{i=1}^{\infty}\in  C_0^{\infty}(\Sigma)$ and $R_i\in C^{\infty}([0,L]; C_0^{\infty}(\Sigma))$. Define the trajectory $(\bar{x}_2^i(\tau;x),\bar{x}_3^i(\tau;x))$ as in \eqref{char}, then there holds
\begin{align*}
    \omega_1^{i} &=\sigma J_0^i(\bar{x}_2^i(0;x),\bar{x}_3^i(0;x))e^{-\int_0^{x_1} \mu_i(\bar{\mathbf{V}})(t;\bar{x}_2^i(t;x),\bar{x}_3^i(t;x)) dt}  \\
    & + \int_0^{x_1} R_i(\tau;\bar{x}_2^i(\tau;x),\bar{x}_3^i(\tau;x))
    e^{-\int_0^{x_1} \mu_i(\bar{\mathbf{V}})(t;\bar{x}_2^i(t;x),\bar{x}_3^i(t;x)) dt} d\tau
\end{align*}
 and $\omega_1^{i}|_{x_1=\xi}\in C_0^{\infty}(\Sigma)$ for any $\xi\in [0,L]$. By the estimates in Lemma \ref{lemma3-4}, the sequence $\{\omega_1^{i}\}_{i=1}^{\infty}$ is bounded in $H^2(\Omega)$, one can extract a subsequence which is weakly convergent to some function $\omega_1\in H^2(\Omega)$. Similar to the proof of Proposition \ref{lemma3-5}, one can prove that $\omega_1\in L^{\infty}([0,L];H_0^2(\Sigma)) \cap W^{1,\infty}([0,L];H^1_0(\Sigma))$ and $\omega_1$ is the unique solution to \eqref{3-20}.


\par Then it follows from \cite[Chapter 1, Theorem 3.1 ]{JE72} that $\omega_1 \in C([0,L];H^{3/2}(\Sigma))$ and $\p_1 \omega_1 \in C([0,L];H^{1/2}(\Sigma))$ and hence
\begin{equation*}
    \omega_1 \in C_w([0,L];H_0^2(\Sigma)),\p_1 \omega_1 \in C_w([0,L];H_0^1(\Sigma))
\end{equation*}
by \cite[Chapter 3, Lemma 8.1 ]{JE72}. It is therefore that $\omega_1$ satisfies the conditions \eqref{351}.
\end{proof}

\par It follows from Proposition \rm{3.5} and \eqref{3-a} that \eqref{3-20}  has a unique solution $ \om_1\in H^2(\Omega) $ satisfying
\begin{equation*}
    \|\om_1\|_{2,\Omega} \leq C(\sigma \|J_0\|_{2,\Sigma}+ \|R\|_{2,\Omega})\leq C\sigma,
\end{equation*}
where $C=C(L,\Sigma,\|B_0,K_0\|_{3,\Sigma},\|J_0\|_{2,\Sigma})$. Then it follows from \eqref{w2} and \eqref{w3} that
\begin{align}\label{3-28}
    \om_2 = \frac{\bar{V}_2\om_1 + \p_3V_4 - \frac{1}{\gamma-1}H^{\gamma-1}(\bar{\mathbf{V}})\p_3V_5}{u_0 + \bar{V}_1},\\ \label{3-29}
    \om_3 = \frac{\bar{V}_3\om_1 - \p_2V_4 + \frac{1}{\gamma-1}H^{\gamma-1}(\bar{\mathbf{V}})\p_2V_5}{u_0 + \bar{V}_1}.
\end{align}
\par Then \(\om_2\) and \(\om_3\) also belong to \(H^2(\Omega)\) with the estimate
\be\label{om23}
\sum_{j=2}^3\|\om_j\|_{2,\Om}\leq C\bigg(\sum_{j=2}^3\|\bar{V}_j\|_{2,\Om}\|\om_1\|_{2,\Om}+\sum_{j=4}^5\|V_j\|_{3,\Omega}\bigg)\leq C\sigma
\ee
where  $C = C(L,\Sigma,\|(B_{0},K_{0})\|_{3,\Sigma},\|J_0\|_{2,\Sigma})$.

We verify that
\begin{eqnarray}\label{2332}
    \begin{aligned}
        \quad n_2 \omega_3 - n_3 \omega_2
         &= \frac{1}{\bar u + \bar V_1}\left( n_2 \bar V_3 \om_1 -n_2\p_2 V_4 - \frac{1}{\gamma-1}(H(\bar{\mathbf{V}}))^{\gamma - 1} n_2\p_2V_5 \right. \\
         & \left. - n_3 \bar V_2 \om_1 -n_3\p_3 V_4 + \frac{1}{\gamma-1}(H(\bar{\mathbf{V}}))^{\gamma-1} n_3\p_3V_5 \right) \\
        & = 0,\quad on \; \p\Gamma_0 \cup \p\Gamma_L
    \end{aligned}
\end{eqnarray}
since $ \omega_1|_{x_1=0,L} \in H_0^2(\Sigma)$ and $(V_4,V_5)|_{x_1=0,L} \in H_0^3(\Sigma)$.\\

Moreover, we also verify that $\sum_{j=1}^3\p_{x_j}\om_j = 0$. Indeed, one has
\begin{align*}
    &\quad \p_1\om_1+\p_2\om_2+\p_3\om_3 \\
    &= \p_1\om_1 + \p_2\left(\frac{\bar{V}_2\om_1 + \p_3V_4 - \frac{1}{\gamma-1}(H(\bar{\mathbf{V}}))^{\gamma-1}\p_3V_5}{u_0 + \bar{V}_1}\right)\\
    &\quad\quad+ \p_3\left(\frac{\bar{V}_3\om_1 - \p_2V_4 - \frac{1}{\gamma-1}(H(\bar{\mathbf{V}}))^{\gamma-1}\p_2V_5}{u_0 + \bar{V}_1}\right)\\
    & =\p_1\om_1 + \frac{\bar{V}_2}{u_0 + \bar{V}_1}\p_2\om_1 + \frac{\bar{V}_3}{u_0 + \bar{V}_1}\p_3\om_1 + \left[\p_2\left(\frac{\bar{V}_2}{u_0 + \bar{V}_1}\right) + \p_3\left(\frac{\bar{V}_3}{u_0 + \bar{V}_1}\right)\right]\om_1\\
    & \quad\quad+ \bigg[ \p_2\left(\frac{1}{u_0 + \bar{V}_1}\right)\p_3V_4-\p_3\left(\frac{1}{u_0 + \bar{V}_1}\right)\p_2 V_4 + \frac{1}{\gamma-1}\p_3\left(\frac{(H(\bar{\mathbf{V}}))^{\gamma-1}}{u_0 + \bar{V}_1}\right)\p_2 V_5\\
    &\quad\quad-\frac{1}{\gamma-1}\p_2\left(\frac{(H(\bar{\mathbf{V}}))^{\gamma-1}}{u_0 + \bar{V}_1}\right)\p_3 V_5\bigg]=0
\end{align*}
which gives us that
\be\label{23}
(\p_2 \omega_2 +\p_3 \omega_3)|_{x_1 = 0,L} = -\p_1 \omega_1|_{x_1 = 0,L} \in H_0^1(\Sigma)\subset H_{00}^{1/2}(\Sigma).
\ee

\par{\bf Step 3. Resolving the velocity field.}
\par In this step, we solve the following linearized deformation-curl system:
\begin{align}\label{3-30}
    \begin{cases}
        (1-\bar M^2)\p_1V_1 + \p_2V_2 + \p_3V_3 = F(\bar{\mathbf{V}},\nabla\bar{\mathbf{V}}),\\
        \p_2V_3 - \p_3V_2 = \om_1,\\
        \p_3V_1 - \p_1V_3 = \om_2,\\
        \p_1V_2 - \p_2V_1 = \om_3,\\
        V_1(0,x^{\prime}) = g_1(\bar{\mathbf{V}})(0,x^{\prime}),\\
        V_1(L,x^{\prime}) = g_2(\bar{\mathbf{V}},V_4,V_5)(L,x^{\prime})+ \kappa ,\\
        n_2V_2 + n_3V_3 = 0 \quad \text{on} \; \Gamma_{\om},
    \end{cases}
\end{align}
where
\begin{equation*}
    \kappa= \frac{1}{|\Sigma|}\left(\int_{\Sigma} g_1{x^{\prime} dx^{\prime} }  -  \int_{\Sigma} g_2{x^{\prime} dx^{\prime} }   + \int_{\Omega}  \frac{1}{1-\bar M^2}F(\bar{\mathbf{V}},\nabla\bar{\mathbf{V}})dx    \right),
\end{equation*}
\begin{equation*}
    \begin{aligned}
      &F(\bar{\mathbf{V}},\nabla\bar{\mathbf{V}}) = \left( \frac{(\bar u + \bar{V}_1)^2}{c^2(H(\ol{\mathbf{V}}),\bar K + \bar{V}_5)} - \bar{M}^2 \right)\p_1\bar{V}_1 + \sum_{j=2}^3 \frac{(\bar u + \bar{V}_1)\bar{V_j}}{c^2(H(\ol{\mathbf{V}}),\bar K + \bar{V}_5)}\p_1\bar{V}_j \\
    &\quad \quad+ \sum_{i,j=2}^3 \frac{\bar{V}_i\bar{V}_j}{c^2(H(\ol{\mathbf{V}}),\bar K + V_5)}\p_i\bar{V}_j,
    \end{aligned}
\end{equation*}
\begin{equation*}
    \begin{aligned}
& g_1(\bar{\mathbf{V}})(x')
     =\frac{\sigma m_0}{\bar \rho(1-\bar M^2)}-\frac{\sigma \bar u B_0}{c^2(\bar \rho,\bar K)-\bar u^2}+\frac{\sigma\bar u K_0}{(\gamma -1)\bar K(1-\bar M^2)}\\
     &\quad - \frac{\bar{V}_1}{\bar\rho(1-\bar M^2)}\bigg(H(\bar B+\sigma B_0,\bar K+\sigma K_0,(\bar u+\bar{V}_1)^2 + \bar{V}_2^2+\bar{V}_3^2)-H(\bar B,\bar K,\bar u^2)\bigg)(0,x')\\
     &\quad - \frac{\bar u}{\bar\rho(1-\bar M^2)}\bigg(H(\bar B+\sigma B_0,\bar K+\sigma K_0,(\bar u+\bar{V}_1)^2 + \bar{V}_2^2+\bar{V}_3^2)-H(\bar B,\bar K,\bar u^2)\\
     &\quad  -\frac{\sigma \bar\rho B_0}{c^2(\bar\rho,\bar K)} + \frac{\sigma \bar\rho K_0}{(\gamma-1)\bar K} +\bar M^2 \bar{V}_1\bigg)(0,x'),
    \end{aligned}
\end{equation*}
\begin{equation*}
\begin{aligned}
&g_2(\bar{\mathbf{V}},V_4,V_5)(x')
    =\frac{\sigma m_{L}}{\bar\rho(1-\bar M^2)} -\frac{\bar u V_4(L,x^{\prime})}{c^2(\bar\rho,\bar K) - \bar u^2}+\frac{\bar uV_5(L,x^{\prime})}{(\gamma-1)\bar K(1-\bar M^2)}\\
   &\quad - \frac{\bar{V}_1}{\bar\rho(1-\bar M^2)}\bigg(H(\bar B+\bar{V}_4,\bar K+\bar{V}_5,(\bar u+\bar{V}_1)^2 + \bar{V}_2^2+\bar{V}_3^2)-H(\bar B,\bar K,\bar u^2)\bigg)(L,x')\\
     &\quad - \frac{\bar u}{\bar\rho(1-\bar M^2)}\bigg(H(\bar B+\bar{V}_4,\bar K+\bar{V}_5,(\bar u+\bar{V}_1)^2 + \bar{V}_2^2+\bar{V}_3^2)-H(\bar B,\bar K,\bar u^2)\\
     &\quad  -\frac{\bar\rho \bar{V}_4}{c^2(\bar\rho,\bar K)} + \frac{ \bar\rho\bar{V}_5}{(\gamma-1)\bar K} +\bar M^2 \bar{V}_1 \bigg)(L,x'),\\
\end{aligned}
\end{equation*}
here the constant $\kappa$ is introduced to guarantee that the solvability condition for the deformation-curl system with normal boundary condition holds.

\par We use the Duhamel's principal to solve the problem \eqref{3-30}. First, consider the standard div-curl system with homogeneous normal boundary conditions:
\begin{align}\label{3-31}
    \begin{cases}
        \p_1\dot{V_1} + \p_2\dot{V_2} + \p_3\dot{V_3} = 0,\\
        \text{curl }{\bf \dot{V}}= (\om_1,\om_2,\om_3)^t,\\
        \dot{V_1}(0,x') = \dot{V_1}(L,x') = 0,\\
        n_2\dot{V_2} + n_3\dot{V_3} = 0 \quad \text{on} \; \Gamma_{\om}.
    \end{cases}
\end{align}

\par Since the domain $\Om$  is a cylinder, so the boundary condition needs to be interpreted with some care. Since we seek a solution in \(H^3(\Om) \subset C^1(\Bar{\Omega})\) the meaning of the condition is clear on \(\Gamma_0\), \(\Gamma_L\) and \(\p \Sigma \times (0,L)\). For the edges we need the boundary conditions to be compatible so we interpret \(\dot{\mathbf{V}}\cdot \bf n\)  as
\begin{equation*}
    \begin{cases}
       {\bf n}' \cdot \dot{\mathbf{V}}'= 0,\\
        \dot{V_1} = 0,
    \end{cases}
\end{equation*}
on \( \p \Gamma_0\) and \(\p \Gamma_L\), where $\dot{\mathbf{V}}'= (\dot{V_2},\dot{V_3})$ and $ {\bf n}'=(n_2,n_3)$ is the normal to $\p \Sigma$.\\[2ex]

\par The unique solvability of \eqref{3-31} on the Soblev space $H^3(\Omega)$ had been proved in \cite[Theorem  \rm{4.19}]{SS16} by using the vector potential formulation of the div-curl system and the separation of variables.

\begin{lemma}\label{lemma3-6}

Assume that the vector field $(\omega_1,\omega_2,\omega_3)\in (H^2(\Omega))^3$ satisfies the divergence free condition $\displaystyle\sum_{i=1}^3 \p_i \omega_i=0$ in $\Omega$, $n_2\om_3 - n_3\om_2 = 0$ on $\p\Gamma_0$ and $\p\Gamma_L$ and  $\p_1 \om_1|_{x_1=0},\p_1 \om_1|_{x_1=L} \in H^{1/2}_{00}(\Sigma)$, then there exists a unique solution $\dot{\mathbf{V}}$ to \eqref{3-31} in $H^3(\Omega)$ satisfying the following estimates
\begin{equation*}
    \|\dot{\mathbf{V}}\|_{1,\Omega} \leq C \sum_{i=1}^3\|{\omega}_i\|_{0,\Omega},
\end{equation*}
\text{and}
\begin{equation*}
\begin{aligned}
    \|\dot{\mathbf{V}}\|_{3,\Omega} &\leq C\left(\sum_{i=1}^3\|{\omega}_i\|_{2,\Omega}+ \|\p_1 \om_1|_{x_1=0}\|_{H^{1/2}_{00}(\Gamma_0)} +  \|\p_1 \om_1|_{x_1=L}\|_{H^{1/2}_{00}(\Gamma_L)}\right).
    \end{aligned}
\end{equation*}
 \text{where} $C= C(\Omega)$.
\end{lemma}

\par Since $\sum_{j=1}^3\p_j\om_j = 0$, $\om_j \in H^2(\Omega)$ for all $j=1,2,3$ and \eqref{2332},\eqref{23}, all the conditions in Lemma \ref{lemma3-6}, there exists a unique vector field \(\dot{\mathbf{V}} \in H^3(\Omega)\) such that
\begin{equation*}\label{v1}
    \|(\dot{V}_1,\dot{V}_2,\dot{V}_3)\|_{3,\Omega} \leq C\sum_{i=1}^3\|{\omega}_i\|_{2,\Omega}.
\end{equation*}

\par  Let \(\bm W = \bm V- \dot{\bm V}\). Then $ \bm W $ satisfies
\begin{align}\label{3-32}
    \begin{cases}
        (1-\bar M^2)\p_1W_1 + \p_2W_2 + \p_3W_3 = F(\bar{{\bf V}},\nabla \bar{{\bf V}})+ \bar M^2\p_1\dot{V_1},\\
        \text{curl }{\bm W}=0,\\
        W_1(0,x^{\prime}) = g_1(x^{\prime}),\\
        W_1(L,x^{\prime}) = g_2(x^{\prime})+\kappa,\\
        n_2W_2 + n_3W_3 = 0 \quad \text{on} \; \Gamma_{\om}.
    \end{cases}
\end{align}
Since $\dot{V}_1(0,x')=\dot{V}_1(L,x')=0$ for any $x'\in\Sigma$, the following solvability condition for \eqref{3-32} holds:
\begin{align}\label{3-33}
    \int_{\Sigma}g_1(x^{\prime})dx^{\prime}+\int_0^L \int_{\Si}\frac{1}{1-\bar{M}^2}( F(\bar{{\bf V}},\nabla \bar{{\bf V}})+ \bar M^2\p_1\dot{V_1}) dx' dx_1=\int_{\Sigma}g_2(x^{\prime}) + \kappa\;dx^{\prime}.
\end{align}
Since \( \text{curl}\ \bm W = 0\) in \(\Omega\), there exists a potential function \(\phi(x)\) with mean zero on $\Omega$ such that $\bm W = \nabla\phi(x)$. Then $\phi$ solves
\begin{eqnarray}\label{po1}
    \begin{cases}
        (1-\bar M^2)\p_1^2\phi + \p_2^2\phi + \p_3^2\phi = \Tilde{f},\\
        \p_1\phi(0,x^{\prime}) = g_1(x^{\prime}),\p_1\phi(L,x^{\prime}) = g_2(x^{\prime})+\kappa, \quad \forall x^{\prime} \in \Sigma\\
        \frac{\p \phi}{\p n}(x,x^{\prime}) = 0 ,\quad \text{on} \; \Gamma_{\om} \\
        \int_0^L \int_{\Sigma} \phi(x_1,x') dx' dx_1=0
    \end{cases}
\end{eqnarray}
where $\tilde{f}(x_1,x')= F(\bar{{\bf V}},\nabla \bar{{\bf V}})+ \bar M^2\p_1\dot{V_1} $. It is well-known that the following eigenvalue problem
\begin{equation*}
    \begin{cases}
        -(\p_2^2 + \p_3^2)e(x^{\prime}) = \lambda e(x^{\prime}), \quad &\forall x^{\prime}\in \Sigma \\
        (n_2\p_2 + n_3\p_3)e(x^{\prime}) = 0, \quad  &\text{on} \; x^{\prime}\in \p\Sigma
    \end{cases}
\end{equation*}
has a family of eigenvalues $ 0= \lambda_0 < \lambda_1 \leq \lambda_2 \leq ... $  and $ \lambda_n \to +\infty$  as $n \to +\infty$ with corresponding eigenvectors $(e_n(x'))_{n=0}^{\infty}$ which constitute an orthonormal basis in $L^2(\Sigma)$. The eigenvector associated with $\la_0=0$ is $e_0(x^{\prime}) = \frac{1}{|\Sigma|^{1/2}}$.

We express $\phi(x_1,x')$ as generalized Fourier series
\begin{equation}\label{four1}
    \phi(x_1,x') = \sum_{n=0}^{\infty} a_n(x_1)e_n(x^{\prime}).
\end{equation}
Substituting \eqref{four1} into \eqref{po1} yields
\begin{eqnarray}\label{an}
    \begin{cases}
        \left( 1-\bar M^2\right)a_n''(x_1) - \lambda_n a_n(x_1) = \tilde{f}_n(x_1),\\
        a_n'(0) = b_{1,n},\ \ \  a_n'(L) = b_{2,n},
    \end{cases}
\end{eqnarray}
for every $n$, where
\begin{equation*}
b_{1,n} = \int_{\Sigma}g_1(x^{\prime})e_n(x^{\prime})dx^{\prime},\ b_{2,n} = \int_{\Sigma}g_2(x^{\prime})e_n(x^{\prime}) + \kappa e_n(x^{\prime}) dx^{\prime}, \ \Tilde{f}_n(x_1)= \int_{\Sigma}\tilde{f}(x_1,x')e_n(x')dx'.
\end{equation*}
\par For $ n = 0,$ where $\lambda_0 = 0$, the equation becomes
\begin{equation}\label{3-34}
    \begin{cases}
        \left( 1-\bar M^2\right)a_0^{\prime\prime}(x_1) = \Tilde{f}_0(x_1),\\
        a_0^{\prime}(0) = b_{1,0} = \frac{1}{|\Sigma|^{1/2}} \int_{\Sigma}g_1(x^{\prime})dx^{\prime},\\
        a_0^{\prime}(L) = b_{2,0} = \frac{1}{|\Sigma|^{1/2}} \int_{\Sigma}g_2(x^{\prime})+ \kappa \;dx^{\prime},\\
        \int_0^L a_0(x_1) dx_1=0.
    \end{cases}
\end{equation}
Thanks to \eqref{3-33}, the solvability condition for \eqref{3-34} holds, which guarantees the existence and uniqueness of solutions to the equation  \eqref{3-34}. Indeed, one has
\begin{eqnarray}\nonumber
    &&a_0(x_1)= \int_0^{x_1} \left( \frac{1}{|\Sigma|^{1/2}}\int_{\Sigma}g_1(x^{\prime})dx^{\prime} + \int_0^t \frac{1}{1-\bar M^2}\Tilde{f}_0(s)ds  \right)dt\\\nonumber
    &\quad&\quad\quad\quad-\frac{1}{L} \int_0^L \int_0^{x_1} \left( \frac{1}{|\Sigma|^{1/2}}\int_{\Sigma}g_1(x^{\prime})dx^{\prime} + \int_0^t \frac{1}{1-\bar M^2}\Tilde{f}_0(s)ds  \right)dt dx_1.
\end{eqnarray}
For $n > 0 $, the corresponding homogeneous equation becomes
\begin{equation}\label{3-35}
    \begin{cases}
        \left( 1-\bar M^2\right)a_n''(x_1) - \lambda_n a_n(x_1) = 0,\\
        a_n'(0) = a_n'(L) = 0.
    \end{cases}
\end{equation}
Since $\la_n>0$ for any $n>0$, by the maximum principle and Hopf's Lemma, the equation \eqref{3-35} has only zero solution. Then by the Fredholm Alternative Theorem to second order elliptic equation, the equation \eqref{an} for $n>0$ has a unique classical solution $a_n$. Decompose $a_n=c_n(x_1)+d_n(x_1)$, where the function $c_n(x_1)$ solves the equation \eqref{an} with homogeneous boundary data and $d_n(x_1)$ solves the
\begin{equation}\nonumber
    \begin{cases}
        \left( 1-\bar M^2\right)d_n'(x_1) - \lambda_n d_n(x_1) = 0,\\
        d_n^{\prime}(0) = b_{1,n},\ \ \ \ d_n^{\prime}(L) = b_{2,n}.
    \end{cases}
\end{equation}
Integration by parts yields that for $\tilde{\lambda}_n =\frac{\lambda_n}{1-\bar M^2}$
\be\no
&&\int_0^L \tilde{f}_n^2(x_1)dx_1=\int_0^L (c_n''-\tilde{\lambda}_n c_n)^2 dx_1=\int_0^L (|c_n''|^2+2\tilde{\lambda}_n |c_n'|^2+ \tilde{\lambda}_n^2 |c_n|^2 dx_1,\\\no
&&\int_0^L (|c_n^{(4)}|^2+\tilde{\lambda}_n |c_n^{(3)}|^2 d x_1=\int_0^L (\tilde{\lambda}_n c_n''+\tilde{f}_n'')^2+ \tilde{\lambda}_n(\tilde{\lambda}_n c_n'+ \tilde{f}_n')^2 dx_1\\\no
&&\leq \int_0^L \tilde{\lambda}_n^2 |c_n''|^2+ \tilde{\lambda}_n^3 |c_n'|^2+ |\tilde{f}_n''|^2+\tilde{\lambda}_n |\tilde{f}_n'|^2) dx_1
\ee
and thus
\be\no
\int_0^L (|c_n^{(4)}|^2+\tilde{\lambda}_n |c_n^{(3)}|^2+\tilde{\lambda}_n^2 |c_n''|^2+ \tilde{\lambda}_n^3 |c_n'|^2+\tilde{\lambda}_n^4 |c_n|^2) d x_1\lesssim \int_0^L (|\tilde{f}_n''|^2+\tilde{\lambda}_n |\tilde{f}_n'|^2+\tilde{\lambda}_n^2 |\tilde{f}_n|^2) dx_1.
\ee

For the functions $d_n$, there is an explicit formulae
\begin{equation*}
    d_n(x_1) = -b_{1,n}\frac{cosh\left(\sqrt{\tilde{\lambda}_n}(x_1 -L)\right)}{\sqrt{\tilde{\lambda}_n}sinh\left(\sqrt{\tilde{\lambda}_n} L\right)} +
    b_{2,n}\frac{cosh\left(\sqrt{\tilde{\lambda}_n}x_1 \right)}{\sqrt{\tilde{\lambda}_n}sinh\left(\sqrt{\tilde{\lambda}_n} L\right)}.
\end{equation*}
Squaring and integrating gives
\be\no
&&\int_0^L (|d_n^{(4)}|^2+\tilde{\lambda}_n |d_n^{(3)}|^2 +\tilde{\lambda}_n^2 |d_n''|^2+3\tilde{\lambda}_n^3 |d_n'|^2+ 2\tilde{\lambda}_n^4 |d_n|^2) d x_1 \\\no
&&=2\int_0^L (\tilde{\lambda}_n^2 |d_n''|^2 + 2\tilde{\lambda}_n^3 |d_n'|^2 +\tilde{\lambda}_n^4 |d_n|^2)dx_1\\\no
&&=2\tilde{\lambda}_n^2\int_0^L (d_n''-\tilde{\lambda}_n d_n)^2 dx_1 +2\tilde{\lambda}_n^3 (d_n(L)d_n'(L)-d_n(0)d_n'(0))\\\no
&&=2\tilde{\lambda}_n^{5/2}\frac{-2b_{1,n}b_{2,n}+(b_{1,n}^2+b_{2,n}^2)\cosh(\sqrt{\tilde{\lambda}_n}L)}{\sinh(\sqrt{\tilde{\lambda}_n}L)}\\\no
&&\lesssim \tilde{\lambda}_n^{5/2}(|b_{1,n}|^2+|b_{2,n}|^2).
\ee
Combining the above estimates, one has
\be\nonumber
&&\|(W_1,W_2,W_3)\|_{3,\Omega}\leq\|\phi\|_{H^4{(\Omega)}}\leq C(\|\tilde{f}\|_{2,\Omega}+\|g_1\|_{\frac{5}{2},\Sigma}+\|g_2+\kappa\|_{\frac{5}{2},\Sigma})\\\no
&&\leq C(\|F\|_{2,\Om}+\|\dot{V}_1\|_{3,\Om}++\|g_1\|_{\frac{5}{2},\Sigma}+\|g_2+\kappa\|_{\frac{5}{2},\Sigma}).
\ee

Note that
\be\no
&&\|g_1\|_{\frac{5}{2},\Sigma} \leq C(\|m_0\|_{\frac{5}{2},\Sigma},\|(B_0,K_0))\|_{3,\Sigma})(\sigma +\sigma\delta +\delta^2),\\\no
&&\|g_2+\kappa\|_{\frac{5}{2},\Sigma} \leq C(\|m_L\|_{\frac{5}{2},\Sigma},\|(B_0,K_0))\|_{3,\Sigma})(\sigma +\sigma\delta +\delta^2)+ C\|(V_4,V_5)\|_{3,\Omega},
\ee
thus
\begin{align} \nonumber
    \sum_{i=1}^3\|V_i\|_{3,\Omega}
    &\leq \sum_{i=1}^3\|W_i\|_{3,\Omega} + \sum_{i=1}^3\|\dot{V}_i\|_{3,\Omega}\\ \nonumber
    &\leq C(\|g_i\|_{\frac{5}{2},\Sigma} + \|F\|_{2,\Omega} + \|\dot{V}\|_{3,\Omega})\leq C(\sigma+\delta^2+\sigma\delta),
\end{align}
where  $C = C(\Sigma,L,\|(B_{0},K_{0})\|_{3,\Sigma},\|J_0\|_{2,\Sigma},\|m_0\|_{\frac{5}{2},\Sigma},\|m_L\|_{\frac{5}{2},\Sigma})  $.

\par In a word, we have constructed a new $ \mathbf{V} \in H^3(\Omega)$ and
\begin{equation}\nonumber
    \sum_{i=1}^5\|V_i\|_{3,\Omega} \leq C_1(\sigma + \sigma\delta + \delta^2).
\end{equation}
Let $ \sigma_1 = \frac{1}{2(2C_1^2 + C_1)}$ and choose $\delta = 2C_1\sigma$ with $\sigma \leq \sigma_1$, then
\begin{equation*}
    \|\mathbf{V}\|_{\Xi_\delta} \leq C_1\sigma + C_1(\delta + \sigma)\delta = \frac{1}{2}\delta + C_1(2C_1\sigma + \sigma)\delta \leq \delta.
\end{equation*}
Hence $ \mathbf{V} \in \Xi_\delta$.  Define a mapping $ \mathcal{J}$ from $\Xi_\delta$ to itself by
\begin{equation*}
    \mathcal{J}(\bar{\mathbf{V}}) = \mathbf{V}.
\end{equation*}
Next we will show that $ \mathcal{J}$ is a contraction in $\Xi_\delta$. Let $\bar{\mathbf{V}}^k \in \Xi_\delta$, $ k =1,2$, we have $\mathbf{V}^k = \mathcal{J}(\bar{\mathbf{V}}^k)$ for $ k = 1,2$. Define
\begin{equation*}
    \mathbf{Y} = \mathbf{V}^1 - \mathbf{V}^2,\quad \bar{\mathbf{Y}} = \bar{\mathbf{V}}^1 - \bar{\mathbf{V}}^2.
\end{equation*}
We first estimate $Y_4$ and $Y_5$ . It follows from \eqref{3-7} and \eqref{3-8} that for $j=4,5$
\begin{equation}\nonumber
    \begin{cases}
        \p_1Y_j + \frac{\bar{V}_2^1}{\bar u + \bar{V}_2^1}\p_2Y_j + \frac{\bar{V}_3^1}{\bar u + \bar{V}_1^1}\p_3Y_j = - \left( \frac{\bar{V}_2^1}{\bar u + \bar{V}_1^1}- \frac{\bar{V}_2^2}{\bar u + \bar{V}_1^2}
 \right)\p_2V_j^2 - \left( \frac{\bar{V}_3^1}{\bar u + \bar{V}_1^1}- \frac{\bar{V}_3^2}{\bar u + \bar{V}_1^2}
 \right)\p_3V_j^2,\\
 Y_j(0,x') = 0,
    \end{cases}
\end{equation}
Then there holds
\be\no
    &&\|Y_4\|_{2,\Omega}+\|Y_5\|_{2,\Omega}\leq C \sum_{j=4}^5 \sum_{i=2}^3 \b\|\b(\frac{\bar{V}_i^1}{\bar u + \bar{V}_1^1}- \frac{\bar{V}_i^2}{\bar u + \bar{V}_1^2}
\bigg)\p_i V_j^2\b\|_{2,\Omega} \\\no
    &&\leq C\|\mathbf{V}^2\|_{3,\Omega}
    \sum_{k=1}^3\|\bar{\mathbf{Y}}_k\|_{2,\Omega}\leq C \delta\|\bar{\mathbf{Y}}\|_{2,\Omega}.
\ee
Define $J_k = \om_k^1 -\om_k^2$ for $ k = 1,2,3$. It follows from \eqref{3-20}, \eqref{3-28} and \eqref{3-29} that
\begin{equation}\nonumber
    \begin{cases}
        \p_1J_1 + \frac{\bar{V}_2^1}{\bar u + \bar{V}_1^1}\p_2J_1 + \frac{\bar{V}_3^1}{\bar u + \bar{V}_1^1}\p_3J_1 + \mu(\bar{\mathbf{V}}^1)J_1 = R(\bar{\mathbf{V}}^1,V_4^1,V_5^1) - R(\bar{\mathbf{V}}^2,V_4^2,V_5^2)\\
         - \left( \frac{\bar{V}_2^1}{\bar u + \bar{V}_1^1} - \frac{\bar{V}_2^2}{\bar u + \bar{V}_1^2}\right)\p_2\om_1^2
        - \left( \frac{\bar{V}_3^1}{\bar u + \bar{V}_1^1} - \frac{\bar{V}_3^2}{\bar u +\bar{V}_1^2}\right)\p_3\om_1^2 -(\mu(\bar{\mathbf{V}}^1 )-\mu(\bar{\mathbf{V}}^2))\om_1^2,\\
        J_1(0,x') = 0,
    \end{cases}
\end{equation}
and
\be\no
&&J_2= \frac{\bar{V}_2^1 + \p_3Y_4 - \frac{1}{\gamma-1} H^{\gamma-1}(\bar{\mathbf{V}}^1)\p_3Y_5}{\bar{u} + \bar{V}_1^1} + \left(
 \frac{\bar{V}_2^1}{\bar{u} + \bar{V}_1^1} -  \frac{\bar{V}_2^2}{\bar{u} + \bar{V}_1^2}  \right)\om_1^2\\ \nonumber
&&\quad\quad+\left(
 \frac{1}{\bar{u} + \bar{V}_1^1} -  \frac{1}{\bar{u} + \bar{V}_1^2}  \right)\p_3V_4^2 - \frac{1}{\gamma-1}\left( \frac{H^{\gamma-1}(\bar{\mathbf{V}}^1)}{\bar{u} + \bar{V}_1^1} - \frac{H^{\gamma-1}(\bar{\mathbf{V}}^2)}{\bar{u} + \bar{V}_1^1}
  \right)\p_3V_5^2,\\ \nonumber
&&J_3= \frac{\bar{V}_3^1 + \p_2Y_4 - \frac{1}{\gamma-1} H^{\gamma-1}(\bar{\mathbf{V}}^1)\p_2Y_5}{\bar{u} + \bar{V}_1^1} + \left(
 \frac{\bar{V}_3^1}{\bar{u} + \bar{V}_1^1} -  \frac{\bar{V}_3^2}{\bar{u} + \bar{V}_1^2}  \right)\om_1^2\\ \nonumber
&&\quad\quad-\left(
 \frac{1}{\bar{u} + \bar{V}_1^1} -  \frac{1}{\bar{u} + \bar{V}_1^2}  \right)\p_2V_4^2 + \frac{1}{\gamma-1}\left( \frac{H^{\gamma-1}(\bar{\mathbf{V}}^1)}{\bar{u} + \bar{V}_1^1} - \frac{H^{\gamma-1}(\bar{\mathbf{V}}^2)}{\bar{u} + \bar{V}_1^1}
  \right)\p_2V_5^2.
\ee
Then one can conclude that
\begin{align*}
    \sum_{i=1}^3\|J_i\|_{1,\Omega} &\leq C \left( \|\bar{\mathbf{Y}}\|_{2,\Omega}\|\mathbf{\om}^2\|_{2,\Omega} + \|(Y_4,Y_5)\|_{2,\Omega}\|\bar{\mathbf{V}}\|_{2,\Omega} +
\|\bar{\mathbf{Y}}\|_{2,\Omega}\|V_4^2,V_5^2\|_{2,\Omega}
 \right)\\ \nonumber
  & \leq C\delta\| \bar{\mathbf{Y}}\|_{2,\Omega}.
\end{align*}
Finally, we estimate $Y_j$, $ j =1,2,3$. By \eqref{3-30}, there holds
\begin{equation}\nonumber
     \begin{cases}
        (1-\bar M^2)\p_1Y_1 + \p_2Y_2 + \p_3Y_3 = F(\bar{\mathbf{V}}^1,\nabla\bar{\mathbf{V}}^1)-F(\bar{\mathbf{V}}^2,\nabla\bar{\mathbf{V}}^2),\\
        \p_2Y_3 - \p_3Y_2 = J_1,\\
        \p_3Y_1 - \p_1Y_3 = J_2,\\
        \p_1Y_2 - \p_2Y_1 = J_3,\\
        Y_1(0,x') = \left(g_1(\bar{\mathbf{V}}^1)-g_1(\bar{\mathbf{V}}^2)\right)(0,x') ,\\
        Y_2(L,x') = \left(g_2(\bar{\mathbf{V}}^1,W_4^1,V_5^1)-g_2(\bar{\mathbf{V}}^2,W_4^2,V_5^2)\right)(L,x') + \kappa^1 -\kappa^2 \\
        (n_2 Y_2 +n_3 Y_3)(x_1,x')=0,\ \ \ \text{on }\ \Ga_w.
    \end{cases}
\end{equation}
Then the following estimate holds
\begin{align*}
    \sum_{j=1}^3\|Y_j\|_{2,\Omega} &\leq \|F(\bar{\mathbf{V}}^1,\nabla\bar{\mathbf{V}}^1)-F(\bar{\mathbf{V}}^2,\nabla\bar{\mathbf{V}}^2)\|_{H^1(\Omega)} + \sum_{k= 1}^3 \|J_k\|_{1,\Omega} \\ \nonumber
    &\quad\quad + \| \left(g_1(\bar{\mathbf{V}}^1)-g_1(\bar{\mathbf{V}}^2)\right)(0,x') \|_{3/2,\Sigma}\\[2ex] \nonumber
    &\quad\quad + \|\left(g_2(\bar{\mathbf{V}}^1,V_4^1,V_5^1)-g_2(\bar{\mathbf{V}}^2,V_4^2,V_5^2)\right)(L,x') + \kappa^1 -\kappa^2 \|_{3/2,\Sigma}\\[2ex] \nonumber
    & \leq C\| {\mathbf{V}}^1, {\mathbf{V}}^2\|_{2,\Omega}\| \bar{\mathbf{Y}}\|_{2,\Omega}\leq C\delta \|\bar{\mathbf{Y}} \|_{2,\Omega}.
\end{align*}
Combining all the above three estimates, we derive that
\begin{equation*}
    \| \mathbf{Y}\|_{2,\Omega} \leq C_2 \delta\| \bar{\mathbf{Y}}\|_{2,\Omega}.
\end{equation*}
Setting $\sigma_0 = \min \{\sigma_1, \frac{1}{3C_1C_2}\}$, for any $\sigma \leq \sigma_0$, $C_2 \delta \leq \frac{2}{3}$, the mapping $\mathcal{J}$ is a contraction with respect to $H^2$ norm, thus there exists a unique fixed point $\mathbf{V} \in \Xi$ of $\mathcal{J}$. Denoting $\textbf{u} = (\bar u + V_1, V_2,V_3)$ and $\rho = H(\bar B+ V_4,\bar K+V_5,|\textbf{u}|^2)$, then \eqref{3-30} becomes
\begin{equation}\nonumber
    \begin{cases}
        \text{div}(\rho \textbf{u}) = 0,\\
        \p_2u_3 - \p_3u_2 = \om_1,\\
        \p_3u_1 - \p_1u_3 = \om_2,\\
        \p_1u_2 - \p_2u_1 = \om_3,\\
        (\rho u_1)(0,x') =\sigma m_0(x'),\\
        (\rho u_1)(L,x') =\sigma m_L(x') + \rho_0(1-\bar M^2)\kappa.
    \end{cases}
\end{equation}
By the compatibility condition \eqref{compm} and the conservation\ of\ mass, we immediately conclude that $ \kappa =0$. Hence the solution
\begin{equation}\nonumber
    (\textbf{u},B,K) = (\bar u + V_1,V_2,V_3,\bar B + V_4,\bar K+V_5),
\end{equation}
is a $H^3(\Omega)$ solution to the system \eqref{euler} with boundary conditions \eqref{1-2}-\eqref{1-5}.

{\bf Acknowledgement.} Weng is supported by National Natural Science Foundation of China 12071359, 12221001.

\end{document}